\documentclass[11pt]{amsart}
\usepackage{amsmath,latexsym,amsfonts,amssymb,amsthm}
\usepackage{geometry}
\usepackage{tablefootnote}
\usepackage{float}
\usepackage{mathrsfs}
\usepackage{graphicx,color}
\usepackage{tikz-cd}
\usepackage[all,cmtip]{xy}
\usepackage{enumitem}
\usepackage{verbatim}
\usepackage{bbm}
\usepackage{bm}
\usepackage{mathtools}
\usepackage[colorlinks=true, linkcolor=red!80!black, urlcolor=purple, citecolor=blue!70!black]{hyperref}
\usepackage{appendix}
\usepackage{cleveref}
\usepackage{longtable}

\numberwithin{equation}{section}
\newtheorem{theorem}{Theorem}[section]
\newtheorem{prop}[theorem]{Proposition}
\newtheorem{lemma}[theorem]{Lemma}

\newtheorem{conj}[theorem]{Conjecture}

\theoremstyle{definition}
\newtheorem{corollary}[theorem]{Corollary}
\newtheorem{definition}[theorem]{Definition}
 \newtheorem{question}[theorem]{Question}

\newtheorem{remark}[theorem]{Remark}
\newtheorem{example}[theorem]{Example}

\newcommand{\bC}{\mathbb{C}}

\newcommand{\bQ}{\mathbb{Q}}
\newcommand{\bR}{\mathbb{R}}

\renewcommand{\arraystretch}{1.2}

\newcommand{\bP}{\mathbb{P}}
\newcommand{\bA}{\mathbb{A}}

\newcommand{\calD}{\mathcal{D}}

\newcommand{\calO}{\mathcal{O}}

\newcommand{\calM}{\mathcal{M}}

\usepackage{graphicx}
\newcommand{\bZ}{\mathbb{Z}}

\newcommand{\Supp}{\textrm{Supp}}

\newcommand{\lct}{{\mathrm{lct}}}

\newcommand{\calC}{\mathcal C}

\newcommand{\bF}{\mathbb{F}}

\begin{document}

\title{Rational unicuspidal plane curves of low degree}

\author[DeVleming]{Kristin DeVleming}
\address{Department of Mathematics and Statistics,
University of Massachusetts, Amherst, MA 01003, USA}
\email{kdevleming@umass.edu}

\author[Singh]{Nikita Singh}
\address{Department of Mathematics and Statistics,
University of Massachusetts, Amherst, MA 01003, USA}
\email{nrsingh@umass.edu}

\begin{abstract}
    We completely classify all plane curves of degree at most $30$ with a unique cuspidal (locally unibranch) singular point and rational normalization in terms of the Newton pairs parameterizing the cusp.  We distinguish between prime and composite degree in the classification and study the relationship between prime or composite degree and the number of distinct topological types of cuspidal singularities.  Motivated by wall-crossing in moduli of curves, we also survey several geometric properties of rational unicuspidal plane curves.  
\end{abstract}

\maketitle 


\section{Introduction}

It is a difficult open question to effectively classify all possible singularities that can appear on a degree $d$ plane curve $C\subset \bP^2_{\bC}$.  In this paper, we explore the classification question for \textit{rational unicuspidal curves} and provide a complete classification of low degree curves. 

\begin{definition}
    A plane curve $C \subset \bP^2_\bC$ is \textbf{cuspidal} if it is locally irreducible at all singular points.  It is \textbf{unicuspidal} if is has a unique singular point and this point is cuspidal.  
\end{definition}

\begin{definition}
    A projective curve $C$ is \textbf{rational} if it has rational normalization $C^\nu \cong \bP^1$.  
\end{definition}

Classifying the allowable cusps on plane curves is related to many difficult conjectures and results in algebraic geometry, such as the Coolidge-Nagata problem, determining the maximal number of cusps on a plane curve, or rigidity questions related to the curve and its complement.  More recently, the classification of such curves has proved necessary for problems in moduli of curves and surfaces.  For instance, rational unicuspidal limits of smooth quintic curves give rise, under semistable reduction, to hyperelliptic genus 6 curves (see, e.g. \cite[\S 7.1]{ADL19} or \cite[Theorem D, Example 1.11]{DS22}), and similar rational unicuspidal plane curves of higher degree can give rise to lower gonality (non-planar) limits of degree $d$ planar curves.  

Geometrically, the rationality of the plane curve corresponds to its complete replacement in semistable reduction.  Given a generically smooth family of genus $g \ge 2$ curves $\calC$ over a pointed curve $0 \in T$ with central fiber $\calC_0$, we may compute the \textit{stable replacement} of $\calC_0$--the associated limit of the family in $\overline{\calM}_g$, the moduli space of stable genus $g$ curves--by resolving the surface $\calD \to \calC$ and running a relative minimal model program on $\calD$ over $T$ (suppressing any base change in the notation).  The minimal model program contracts rational curves in the fiber over $0$.  If $\calC_0$ is a rational unicuspidal plane curve, the central fiber of the resolution $\calD_0$ is a configuration of curves with a rational tail that is the strict transform of $\calC_0$.  This rational tail is contracted as the first step of the minimal model program and $\calC_0$ completely disappears via semistable reduction.  Furthermore, if no rational tails exist in $\calD_0$, the surface $\calD$ is already minimal so is its own minimal model.  In this sense, rational unicuspidal curves are geometrically significant for moduli of curves and the ``starting point'' for semistable reduction.

Their geometric interest is further illustrated in moduli of pairs, which is the motivation for considering rational unicuspidal curves in $\bP^2$.  For each $\alpha \in (0, 1) \cap \bQ$, there exists a projective moduli space $\calM_d^\alpha$ compactifying the space of pairs $(\bP^2, \alpha C_d)$ where $C_d$ is a smooth degree $d \ge 4$ plane curve.  This space is constructed using K-stability for $ \alpha \in (0, \frac{3}{d})$ (\cite{ADL19, Xu23}), the theory of KSB and KSBA stability for $\alpha \in (\frac{3}{d},1)$ (\cite{Kollar23}), or using \cite{BABWILD} for $\alpha = \frac{3}{d}$.  For each $\alpha$, this space parameterizes some notion of stable pairs, but the stability condition changes as $\alpha$ varies, so the objects parameterized by the boundary of the moduli space may change.  By \cite{ADL19, ABIP}, there is a finite collection of values of $\alpha$ where the stability condition changes, meaning that there exist rational numbers $0 < \alpha_1 < \alpha_2 < \dots < \alpha_n < 1$ such that, for $\alpha,\alpha' \in (\alpha_i, \alpha_{i+1})$, $\calM_d^{\alpha} \cong \calM_d^{\alpha'}$ and for each $\alpha_i$ we have a diagram (up to seminormalization): 

\begin{center}
    \begin{tikzcd}
        \calM_d^{\alpha_i - \epsilon} \arrow[rd] \arrow[rr,dashed] & & \calM_d^{\alpha_i+ \epsilon} \arrow[ld] \\ 
        & \calM_d^{\alpha_i} & \\
    \end{tikzcd}
\end{center}

The singularities of curves $C_d$ correspond to certain walls $\alpha_i$.  For example, if $\lct(C_d) = \beta$ for some $\beta \ge \frac{3}{d}$, then $\beta$ is a value corresponding to a wall-crossing, i.e. $\beta = \alpha_i$ for some $i$.  In general, in the wall-crossing for moduli of pairs $(\bP^2, \alpha C_d)$, the rational unicuspidal curves and their intrinsic geometry control behavior of certain walls and surfaces that may arise as degenerations of the ambient space, as the following examples illustrate. 

\begin{example}\label{quintics}
    There exists a rational unicuspidal curve $C_0$ of degree $5$ with a cusp locally analytically of the form $x^2 + y^{13} = 0$ (c.f. Theorem \ref{thm:onepair}.4).  This curve appears in $\calM_5^\alpha$ for $\alpha \ll 1$, but is unstable for $\alpha > \frac{8}{15}$ \cite[\S 7]{ADL19}.  In the associated moduli spaces of quintic plane curves, there is a wall crossing at $\alpha_i = \frac{8}{15}$ where the pair $(\bP^2, (\alpha_i - \epsilon) C_0) $ appears in the moduli space $\calM_5^{\alpha_i - \epsilon}$ but is replaced by $(X_{26}, (\alpha_i+\epsilon) C_h)$ in $\calM_5^{\alpha_i + \epsilon}$.  Here, $X_{26}$ is a partial smoothing of the $\frac{1}{4}(1,1)$ singularity on $\bP(1,4,25)$ and $C_h$ is a hyperelliptic genus 6 curve.  The birational geometry of the replacement is shown in the following diagram, with details left to \cite[\S 5]{KmoduliAGNESnotes}.

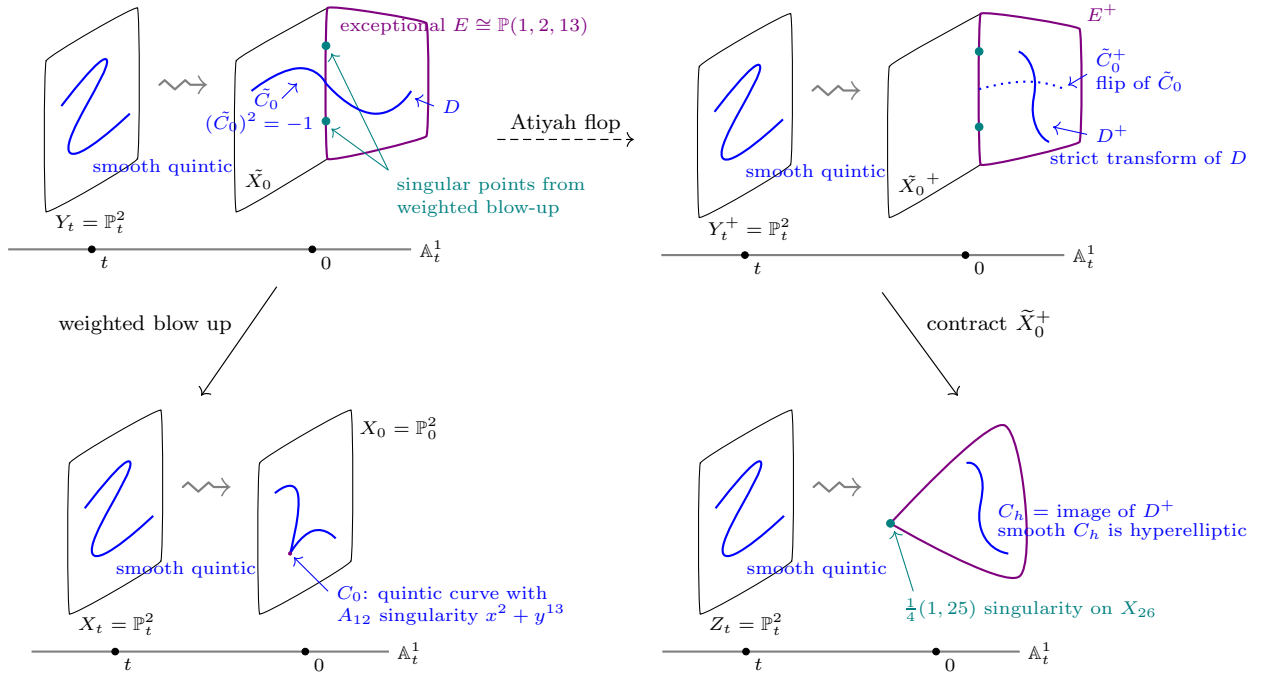
\begin{figure}[h]
\begin{tikzcd}[row sep=huge]
\begin{tabular}{c}
\begin{tikzpicture}[gren0/.style = {draw, circle,fill=greener!80,scale=.7},gren/.style ={draw, circle, fill=greener!80,scale=.4},blk/.style ={draw, circle, fill=black!,scale=.2},plc/.style ={draw, circle, color=white!100,fill=white!100,scale=0.02},smt/.style ={draw, circle, color=gray!100,fill=gray!100,scale=0.02},lbl/.style ={scale=.2}] 

\node[smt] at (-2.5+0,0) (1){};
\node[smt] at (-2.5+0,2) (2){};
\node[smt] at (-2.5+1.2,2.7) (3){};
\node[smt] at (-2.5+1.2,0.7) (4){};
\draw [black] plot [smooth cycle, tension=.1] coordinates { (1) (2) (3) (4) };

\draw [thick,blue] plot [smooth, tension=1] coordinates { (-2.5+.2,1.4) (-2.5+.8,2) (-2.5+.3,0.8) (-2.5+1.1,1.3)};
\node[below right,node font=\tiny,text=blue] at (-2.5+0.5,0.8) {smooth quintic};

\node[smt] at (0,0) (5){};
\node[smt] at (0,2) (6){};
\node[smt] at (1.2,2.7) (7){};
\node[smt] at (1.2,0.7) (8){};
\draw [black] plot [smooth, tension=.1] coordinates { (8) (5) (6) (7) };

\node[smt] at (2.5,2.5) (11){};
\node[smt] at (2.5,1) (12){};
\draw [thick,violet] plot [smooth cycle, tension=.1] coordinates { (8) (12) (11) (7) };
\node[node font=\tiny,color=violet] at (3,2.4) {exceptional $E \cong \bP(1,2,13)$};
\filldraw[color=teal, fill=teal](1.18,2.2) circle (0.05);
\filldraw[color=teal, fill=teal](1.18,1.2) circle (0.05);
\draw[teal,->] (2,.55) -- (1.25,2.1);
\draw[teal,->] (2,.55) -- (1.25,1.1);
\node[below right,node font=\tiny,color=teal] at (2,.55) {singular points from};
\node[below right,node font=\tiny,color=teal] at (2,.25) {weighted blow-up};

\draw [thick,blue] plot [smooth, tension=1] coordinates { (0.2,1.6)  (.8,1.9) (1.18,1.7) };
\draw [thick,blue] plot [smooth, tension=1] coordinates {  (1.18,1.7) (1.8,1.3) (2.3,1.6) };
\draw[blue,->] (.6,1.5) -- (.8,1.7);
\node[node font=\tiny,color=blue] at (.4,1.4) {$\tilde{C_0}$};
\node[node font=\tiny,color=blue] at (.3,1.1) {$(\tilde{C_0})^2= -1$};
\draw[blue,->] (2.6,1.4) -- (2.35,1.5);
\node[right,node font=\tiny,color=blue] at (2.6,1.4) {$D$};

\draw[thick,color=gray] (-2.5-0.5,-.5) to (1.8+0.5,-.5);
\node[blk] at (-2.5+0.6,-.5) (9){};
\node[below right, node font=\tiny] at (9) {$t$};
\node[blk] at (1,-.5) (10){};
\node[below right, node font=\tiny] at (10) {$0$};
\node[right, node font=\tiny] at (1.8+0.5,-.5) {$\bA^1_t$};

\node[below right,node font=\tiny] at (-2.5,0.1) {$Y_t = \bP^2_t$};
\node[above right, node font=\tiny] at (0,0.2) {$\tilde{X_0}$};
\node[node font=\huge, text=gray] at (-2.5+1.8,1.5) {$\rightsquigarrow$};
\end{tikzpicture}
\end{tabular} \arrow[d,swap,start anchor={[xshift=-2ex]},end anchor={[xshift=-8ex]},"\text{weighted blow up}"] \arrow[r,dashed, start anchor={[xshift=-10ex]}, "\text{Atiyah flop}"] &[-2em] \begin{tabular}{c}
\begin{tikzpicture}[gren0/.style = {draw, circle,fill=greener!80,scale=.7},gren/.style ={draw, circle, fill=greener!80,scale=.4},blk/.style ={draw, circle, fill=black!,scale=.2},plc/.style ={draw, circle, color=white!100,fill=white!100,scale=0.02},smt/.style ={draw, circle, color=gray!100,fill=gray!100,scale=0.02},lbl/.style ={scale=.2}] 

\node[smt] at (-2.5+0,0) (1){};
\node[smt] at (-2.5+0,2) (2){};
\node[smt] at (-2.5+1.2,2.7) (3){};
\node[smt] at (-2.5+1.2,0.7) (4){};
\draw [black] plot [smooth cycle, tension=.1] coordinates { (1) (2) (3) (4) };

\draw [thick,blue] plot [smooth, tension=1] coordinates { (-2.5+.2,1.4) (-2.5+.8,2) (-2.5+.3,0.8) (-2.5+1.1,1.3)};
\node[below right,node font=\tiny,text=blue] at (-2.5+0.5,0.8) {smooth quintic};

\node[smt] at (0,0) (5){};
\node[smt] at (0,2) (6){};
\node[smt] at (1.2,2.7) (7){};
\node[smt] at (1.2,0.7) (8){};
\draw [black] plot [smooth, tension=.1] coordinates { (8) (5) (6) (7) };

\node[smt] at (2.5,2.5) (11){};
\node[smt] at (2.5,1) (12){};
\draw [thick,violet] plot [smooth cycle, tension=.1] coordinates { (8) (12) (11) (7) };
\node[node font=\tiny,color=violet] at (2.8,2.6) { $E^+$};
\filldraw[color=teal, fill=teal](1.18,2.2) circle (0.05);
\filldraw[color=teal, fill=teal](1.18,1.2) circle (0.05);

\draw [thick,dotted,blue] plot [smooth, tension=1] coordinates {  (1.18,1.7) (1.8,1.8) (2.3,1.7) };
\draw [thick,blue] plot [smooth, tension=1] coordinates {  (1.7,2.2) (1.9,1.9) (1.9,1.3) (2.1,1) };

\draw[blue,->] (2.6,1.9) -- (2.35,1.75);
\node[right,node font=\tiny,color=blue] at (2.6,2.1) {$\tilde{C}_0^+$};
\node[right,node font=\tiny,color=blue] at (2.6,1.8) {flip of $\tilde{C}_0$};

\draw[blue,->] (2.6,1.1) -- (2.15,1.2);
\node[right,node font=\tiny,color=blue] at (2.6,1.1) {$D^+$};
\node[right,node font=\tiny,color=blue] at (2,0.8) {strict transform of $D$};

\draw[thick,color=gray] (-2.5-0.5,-.5) to (1.8+0.5,-.5);
\node[blk] at (-2.5+0.6,-.5) (9){};
\node[below right, node font=\tiny] at (9) {$t$};
\node[blk] at (1,-.5) (10){};
\node[below right, node font=\tiny] at (10) {$0$};
\node[right, node font=\tiny] at (1.8+0.5,-.5) {$\bA^1_t$};


\node[below right,node font=\tiny] at (-2.5,0.1) {$Y^+_t = \bP^2_t$};
\node[above right, node font=\tiny] at (0,0.2) {$\tilde{X_0}^+$};
\node[node font=\huge, text=gray] at (-2.5+1.8,1.5) {$\rightsquigarrow$};
\end{tikzpicture}
\end{tabular} \arrow[d,start anchor={[xshift=-6ex]},end anchor={[xshift=0ex]},"\text{contract } \widetilde{X}_0^+"]  \\
\begin{tabular}{c}
\begin{tikzpicture}[gren0/.style = {draw, circle,fill=greener!80,scale=.7},gren/.style ={draw, circle, fill=greener!80,scale=.4},blk/.style ={draw, circle, fill=black!,scale=.2},plc/.style ={draw, circle, color=white!100,fill=white!100,scale=0.02},smt/.style ={draw, circle, color=gray!100,fill=gray!100,scale=0.02},lbl/.style ={scale=.2}] 

\node[smt] at (-2.5+0,0) (1){};
\node[smt] at (-2.5+0,2) (2){};
\node[smt] at (-2.5+1.2,2.7) (3){};
\node[smt] at (-2.5+1.2,0.7) (4){};
\draw [black] plot [smooth cycle, tension=.1] coordinates { (1) (2) (3) (4) };

\draw [thick,blue] plot [smooth, tension=1] coordinates { (-2.5+.2,1.4) (-2.5+.8,2) (-2.5+.3,0.8) (-2.5+1.1,1.3)};
\node[below right,node font=\tiny,text=blue] at (-2.5+0.5,0.8) {smooth quintic};

\node[smt] at (0,0) (5){};
\node[smt] at (0,2) (6){};
\node[smt] at (1.2,2.7) (7){};
\node[smt] at (1.2,0.7) (8){};
\draw [black] plot [smooth cycle, tension=.1] coordinates { (5) (6) (7) (8) };

\draw [thick,blue] plot [smooth, tension=1] coordinates { (0.2,1.6) (0.5,1.6) (0.4,.8) };
\draw [thick,blue] plot [smooth, tension=1] coordinates {  (0.4,.8) (.7,1.1) (1,1) };
\draw[blue,->] (.9,.3) -- (.45,.75);
\filldraw[color=violet, fill=violet](0.4,0.8) circle (0.02);
\node[right,color=blue,node font=\tiny] at (.9,.3) {$C_0$: quintic curve with};
\node[right,color=blue,node font=\tiny] at (.9,0) {$A_{12}$ singularity $x^2 +y^{13}$};

\draw[thick,color=gray] (-2.5-0.5,-.5) to (1.2+0.5,-.5);
\node[blk] at (-2.5+0.6,-.5) (9){};
\node[below right, node font=\tiny] at (9) {$t$};
\node[blk] at (0.6,-.5) (10){};
\node[below right, node font=\tiny] at (10) {$0$};
\node[right, node font=\tiny] at (1.2+0.5,-.5) {$\bA^1_t$};

\node[below right,node font=\tiny] at (-2.5,0.1) {$X_t = \bP^2_t$};
\node[below right, node font=\tiny] at (7) {$X_0 = \bP^2_0$};
\node[node font=\huge, text=gray] at (-2.5+1.8,1.5) {$\rightsquigarrow$};
\end{tikzpicture}
\end{tabular} & \begin{tabular}{c}
\begin{tikzpicture}[gren0/.style = {draw, circle,fill=greener!80,scale=.7},gren/.style ={draw, circle, fill=greener!80,scale=.4},blk/.style ={draw, circle, fill=black!,scale=.2},plc/.style ={draw, circle, color=white!100,fill=white!100,scale=0.02},smt/.style ={draw, circle, color=gray!100,fill=gray!100,scale=0.02},lbl/.style ={scale=.2}] 

\node[smt] at (-2.5+0,0) (1){};
\node[smt] at (-2.5+0,2) (2){};
\node[smt] at (-2.5+1.2,2.7) (3){};
\node[smt] at (-2.5+1.2,0.7) (4){};
\draw [black] plot [smooth cycle, tension=.1] coordinates { (1) (2) (3) (4) };

\draw [thick,blue] plot [smooth, tension=1] coordinates { (-2.5+.2,1.4) (-2.5+.8,2) (-2.5+.3,0.8) (-2.5+1.1,1.3)};
\node[below right,node font=\tiny,text=blue] at (-2.5+0.5,0.8) {smooth quintic};

\node[smt] at (0,1.2) (6){};
\node[smt] at (1.5,2.5) (7){};
\node[smt] at (1.7,0.5) (8){};
\draw [thick,violet] plot [smooth, tension=0.5] coordinates { (6) (7) (8) (6) };
\filldraw[color=teal, fill=teal](0,1.2) circle (0.05);
\draw[teal,->] (.3,.3) -- (.05,1.1);
\node[below right,color=teal,node font=\tiny] at (.05,.3) {$\frac{1}{4}(1,25)$ singularity on $X_{26}$};

\draw [thick,blue] plot [smooth, tension=1] coordinates {  (1,2) (1.2,1.8) (1.2,1.1) (1.55,.8) };
\node[below right,color=blue,node font=\tiny] at (1.3,1.6) {$C_h$ = image of $D^+$};
\node[below right,color=blue,node font=\tiny] at (1.3,1.3) {smooth $C_h$ is hyperelliptic};

\draw[thick,color=gray] (-2.5-0.5,-.5) to (1.2+0.5,-.5);
\node[blk] at (-2.5+0.6,-.5) (9){};
\node[below right, node font=\tiny] at (9) {$t$};
\node[blk] at (0.6,-.5) (10){};
\node[below right, node font=\tiny] at (10) {$0$};
\node[right, node font=\tiny] at (1.2+0.5,-.5) {$\bA^1_t$};

\node[below right,node font=\tiny] at (-2.5,0.1) {$Z_t = \bP^2_t$};
\node[node font=\huge, text=gray] at (-2.5+1.8,1.5) {$\rightsquigarrow$};
\end{tikzpicture}
\end{tabular} \\
\end{tikzcd}
\vspace{-.5in}
\caption{Replacement of the $A_{12}$ quintic curve.}
\label{f:quinticsecondwall}
\end{figure}

The birational operations to produce the pair $(X_{26}, C_h)$ are explicitly tied to geometric properties of the unicuspidal plane curve: first, $C_0$ must have been \textit{rational} with minimal resolution $\widetilde{C}_0$ such that $\widetilde{C}_0$ has \textit{negative self-intersection} in the surface $\widetilde{X}_0$ in order to be flopped from the top left diagram to the top right diagram.  Additionally, the complement $\bP^2 - C_0$ must have had \textit{negative Kodaira dimension} to be contracted to a point in the minimal model, producing the normal surface $X_{26}$. These properties will be explored in general in Section \ref{s:kodairadim}.

Furthermore, starting with a sufficiently general family of smooth quintics degenerating to $C_0$, we see that the curve $D$ in the exceptional divisor of the first weighted blow-up is smooth and itself is isomorphic to the stable replacement of $C_0$.  So, the geometry of $C_0$ in this example gives rise to interesting smooth degenerations of quintic plane curves and a normal degeneration of $\bP^2$.  In fact, this can be generalized for any unicuspidal rational curve as in Theorem \ref{thm:onepair}.4, with $2, 5, 13$ replaced by the appropriate Fibonacci numbers.  
\end{example}

\begin{example}\label{octics}
    There exists a rational octic curve $C_0$ with a unique cusp of the form $y^{3} = x^{22}$ (c.f. Theorem \ref{thm:onepair}.5).  Consider a general family of smooth octics $C_t$ degenerating to $C_0$ over $\bA^1_t$.  The log canonical threshold of the curve $C_0$ is $\frac{1}{3} + \frac{1}{22} < \frac{1}{2}$, so $(\bP^2, \frac{1}{2}C_0)$ is not log canonical and is not a stable pair in $\calM_8^{1/2}$.  In this case, computing the stable replacement of the limit $(\bP^2, \frac{1}{2}C_0)$ (the associated limit in the KSB moduli space) gives rise to an slc surface of general type. Generically for $t \ne 0$, the fiber $(\bP^2, \frac{1}{2}C_t)$ is log smooth, and the associated double cover of $\bP^2$ branched over $C_t$ is a smooth Horikawa surface of general type.  However, the pair $(\bP^2, \frac{1}{2}C_0)$ is not slc, so the associated branched cover is not slc.  To compute the limit of the family of Horikawa surfaces in the KSB moduli space, we must find the stable replacement of $(\bP^2, \frac{1}{2}C_0)$.  The birational geometry of computing the replacement is indicated in the following diagram. 

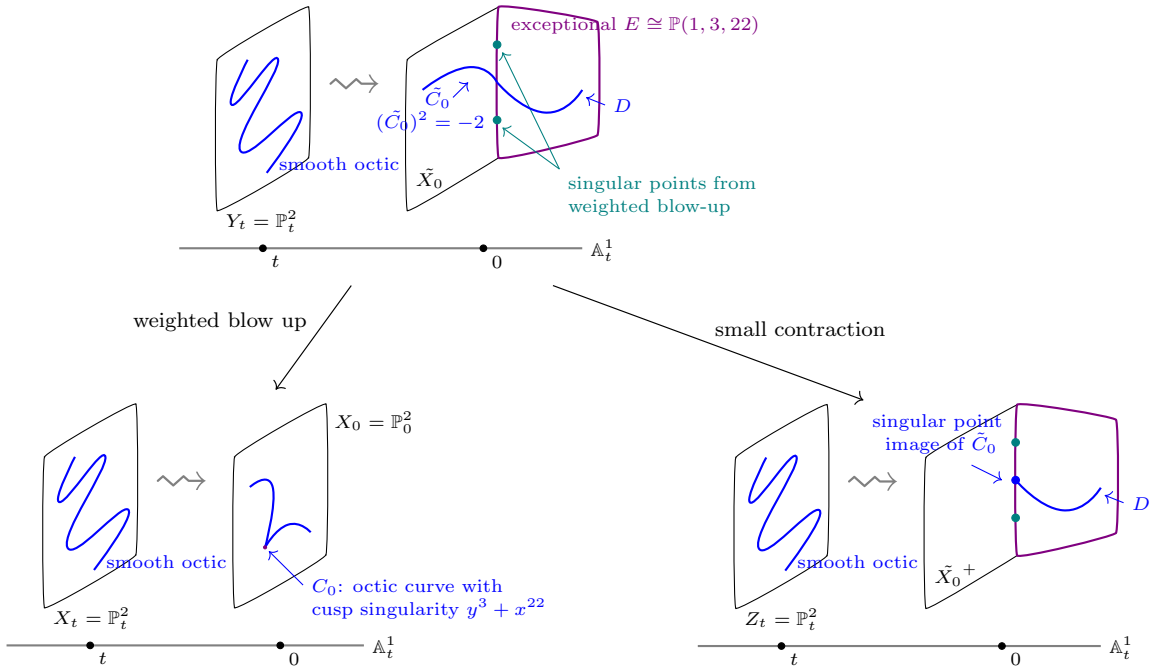
\begin{figure}[h]
\begin{tikzcd}[row sep=huge]
\hspace{1in} \begin{tabular}{c}
\begin{tikzpicture}[gren0/.style = {draw, circle,fill=greener!80,scale=.7},gren/.style ={draw, circle, fill=greener!80,scale=.4},blk/.style ={draw, circle, fill=black!,scale=.2},plc/.style ={draw, circle, color=white!100,fill=white!100,scale=0.02},smt/.style ={draw, circle, color=gray!100,fill=gray!100,scale=0.02},lbl/.style ={scale=.2}] 

\node[smt] at (-2.5+0,0) (1){};
\node[smt] at (-2.5+0,2) (2){};
\node[smt] at (-2.5+1.2,2.7) (3){};
\node[smt] at (-2.5+1.2,0.7) (4){};
\draw [black] plot [smooth cycle, tension=.1] coordinates { (1) (2) (3) (4) };

\draw [thick,blue] plot [smooth, tension=1] coordinates {(-2.5+.4,2) (-2.5+.2,1.4) (-2.5+.8,2) (-2.5+.3,0.8) (-2.5+1.1,1.3) (-2.5+.65,0.5)};
\node[below right,node font=\tiny,text=blue] at (-2.5+0.7,0.8) {smooth octic};

\node[smt] at (0,0) (5){};
\node[smt] at (0,2) (6){};
\node[smt] at (1.2,2.7) (7){};
\node[smt] at (1.2,0.7) (8){};
\draw [black] plot [smooth, tension=.1] coordinates { (8) (5) (6) (7) };

\node[smt] at (2.5,2.5) (11){};
\node[smt] at (2.5,1) (12){};
\draw [thick,violet] plot [smooth cycle, tension=.1] coordinates { (8) (12) (11) (7) };
\node[node font=\tiny,color=violet] at (3,2.4) {exceptional $E \cong \bP(1,3,22)$};
\filldraw[color=teal, fill=teal](1.18,2.2) circle (0.05);
\filldraw[color=teal, fill=teal](1.18,1.2) circle (0.05);
\draw[teal,->] (2,.55) -- (1.25,2.1);
\draw[teal,->] (2,.55) -- (1.25,1.1);
\node[below right,node font=\tiny,color=teal] at (2,.55) {singular points from};
\node[below right,node font=\tiny,color=teal] at (2,.25) {weighted blow-up};

\draw [thick,blue] plot [smooth, tension=1] coordinates { (0.2,1.6)  (.8,1.9) (1.18,1.7) };
\draw [thick,blue] plot [smooth, tension=1] coordinates {  (1.18,1.7) (1.8,1.3) (2.3,1.6) };
\draw[blue,->] (.6,1.5) -- (.8,1.7);
\node[node font=\tiny,color=blue] at (.4,1.4) {$\tilde{C_0}$};
\node[node font=\tiny,color=blue] at (.3,1.1) {$(\tilde{C_0})^2= -2$};
\draw[blue,->] (2.6,1.4) -- (2.35,1.5);
\node[right,node font=\tiny,color=blue] at (2.6,1.4) {$D$};

\draw[thick,color=gray] (-2.5-0.5,-.5) to (1.8+0.5,-.5);
\node[blk] at (-2.5+0.6,-.5) (9){};
\node[below right, node font=\tiny] at (9) {$t$};
\node[blk] at (1,-.5) (10){};
\node[below right, node font=\tiny] at (10) {$0$};
\node[right, node font=\tiny] at (1.8+0.5,-.5) {$\bA^1_t$};

\node[below right,node font=\tiny] at (-2.5,0.1) {$Y_t = \bP^2_t$};
\node[above right, node font=\tiny] at (0,0.2) {$\tilde{X_0}$};
\node[node font=\huge, text=gray] at (-2.5+1.8,1.5) {$\rightsquigarrow$};
\end{tikzpicture}
\end{tabular} \arrow[d,swap,start anchor={[xshift=-2ex]},end anchor={[xshift=-8ex]},"\text{weighted blow up}"] \arrow[rd, start anchor={[xshift=-5ex]}, end anchor={[yshift=3ex]}, "\text{small contraction}"] & \\
\hspace{-1in} \begin{tabular}{c}
\begin{tikzpicture}[gren0/.style = {draw, circle,fill=greener!80,scale=.7},gren/.style ={draw, circle, fill=greener!80,scale=.4},blk/.style ={draw, circle, fill=black!,scale=.2},plc/.style ={draw, circle, color=white!100,fill=white!100,scale=0.02},smt/.style ={draw, circle, color=gray!100,fill=gray!100,scale=0.02},lbl/.style ={scale=.2}] 

\node[smt] at (-2.5+0,0) (1){};
\node[smt] at (-2.5+0,2) (2){};
\node[smt] at (-2.5+1.2,2.7) (3){};
\node[smt] at (-2.5+1.2,0.7) (4){};
\draw [black] plot [smooth cycle, tension=.1] coordinates { (1) (2) (3) (4) };

\draw [thick,blue] plot [smooth, tension=1] coordinates {(-2.5+.4,2) (-2.5+.2,1.4) (-2.5+.8,2) (-2.5+.3,0.8) (-2.5+1.1,1.3) (-2.5+.65,0.5)};
\node[below right,node font=\tiny,text=blue] at (-2.5+0.7,0.8) {smooth octic};

\node[smt] at (0,0) (5){};
\node[smt] at (0,2) (6){};
\node[smt] at (1.2,2.7) (7){};
\node[smt] at (1.2,0.7) (8){};
\draw [black] plot [smooth cycle, tension=.1] coordinates { (5) (6) (7) (8) };

\draw [thick,blue] plot [smooth, tension=1] coordinates { (0.2,1.6) (0.5,1.6) (0.4,.8) };
\draw [thick,blue] plot [smooth, tension=1] coordinates {  (0.4,.8) (.7,1.1) (1,1) };
\draw[blue,->] (.9,.3) -- (.45,.75);
\filldraw[color=violet, fill=violet](0.4,0.8) circle (0.02);
\node[right,color=blue,node font=\tiny] at (.9,.3) {$C_0$: octic curve with};
\node[right,color=blue,node font=\tiny] at (.9,0) {cusp singularity $y^3 +x^{22}$};

\draw[thick,color=gray] (-2.5-0.5,-.5) to (1.2+0.5,-.5);
\node[blk] at (-2.5+0.6,-.5) (9){};
\node[below right, node font=\tiny] at (9) {$t$};
\node[blk] at (0.6,-.5) (10){};
\node[below right, node font=\tiny] at (10) {$0$};
\node[right, node font=\tiny] at (1.2+0.5,-.5) {$\bA^1_t$};

\node[below right,node font=\tiny] at (-2.5,0.1) {$X_t = \bP^2_t$};
\node[below right, node font=\tiny] at (7) {$X_0 = \bP^2_0$};
\node[node font=\huge, text=gray] at (-2.5+1.8,1.5) {$\rightsquigarrow$};
\end{tikzpicture}
\end{tabular} & [-2em] \hspace{-1in} \begin{tabular}{c}
\begin{tikzpicture}[gren0/.style = {draw, circle,fill=greener!80,scale=.7},gren/.style ={draw, circle, fill=greener!80,scale=.4},blk/.style ={draw, circle, fill=black!,scale=.2},plc/.style ={draw, circle, color=white!100,fill=white!100,scale=0.02},smt/.style ={draw, circle, color=gray!100,fill=gray!100,scale=0.02},lbl/.style ={scale=.2}] 

\node[smt] at (-2.5+0,0) (1){};
\node[smt] at (-2.5+0,2) (2){};
\node[smt] at (-2.5+1.2,2.7) (3){};
\node[smt] at (-2.5+1.2,0.7) (4){};
\draw [black] plot [smooth cycle, tension=.1] coordinates { (1) (2) (3) (4) };

\draw [thick,blue] plot [smooth, tension=1] coordinates {(-2.5+.4,2) (-2.5+.2,1.4) (-2.5+.8,2) (-2.5+.3,0.8) (-2.5+1.1,1.3) (-2.5+.65,0.5)};
\node[below right,node font=\tiny,text=blue] at (-2.5+0.7,0.8) {smooth octic};

\node[smt] at (0,0) (5){};
\node[smt] at (0,2) (6){};
\node[smt] at (1.2,2.7) (7){};
\node[smt] at (1.2,0.7) (8){};
\draw [black] plot [smooth, tension=.1] coordinates { (8) (5) (6) (7) };

\node[smt] at (2.5,2.5) (11){};
\node[smt] at (2.5,1) (12){};
\draw [thick,violet] plot [smooth cycle, tension=.1] coordinates { (8) (12) (11) (7) };
\filldraw[color=teal, fill=teal](1.18,2.2) circle (0.05);
\filldraw[color=teal, fill=teal](1.18,1.2) circle (0.05);

\draw [thick,blue] plot [smooth, tension=1] coordinates {  (1.18,1.7) (1.8,1.3) (2.3,1.6) };
\draw[blue,->] (.6,1.9) -- (1,1.7);
\node[node font=\tiny,color=blue] at (.15,2.4) {singular point};
\node[node font=\tiny,color=blue] at (.2,2.1) {image of $\tilde{C}_0$};
\draw[blue,->] (2.6,1.4) -- (2.35,1.5);
\node[right,node font=\tiny,color=blue] at (2.6,1.4) {$D$};
\filldraw[color=blue, fill=blue](1.18,1.7) circle (0.05);


\draw[thick,color=gray] (-2.5-0.5,-.5) to (1.8+0.5,-.5);
\node[blk] at (-2.5+0.6,-.5) (9){};
\node[below right, node font=\tiny] at (9) {$t$};
\node[blk] at (1,-.5) (10){};
\node[below right, node font=\tiny] at (10) {$0$};
\node[right, node font=\tiny] at (1.8+0.5,-.5) {$\bA^1_t$};


\node[below right,node font=\tiny] at (-2.5,0.1) {$Z_t = \bP^2_t$};
\node[above right, node font=\tiny] at (0,0.2) {$\tilde{X_0}^+$};
\node[node font=\huge, text=gray] at (-2.5+1.8,1.5) {$\rightsquigarrow$};
\end{tikzpicture}
\end{tabular}   \\
\end{tikzcd}
\vspace{-.5in}
\caption{Replacement of the rational octic curve with cusp $y^3 = x^{22}$.}
\label{f:octics}
\end{figure}

    If $f: Y \to Z$ is the small contraction of the threefolds illustrated above and $\Delta$ is the divisor giving the family of curves in $Y$ (with image $\Delta_Z$ in $Z$), note that $(K_Y + \frac{1}{2}\Delta) \cdot \widetilde{C}_0 = 1 + \frac{1}{2}(-1) = 0$, so $K_Z + \frac{1}{2} \Delta_Z$ is Cartier and relatively ample over $\bA^1_t$ by computation.  Therefore, the replacement of $(\bP^2, \frac{1}{2}C_0)$ in $\calM_8^{1/2}$ is the non-normal surface $Z_0 = \widetilde{X}_0^+ \cup E$.  The pair $(Z_0, \frac{1}{2}D)$ is log canonical and the double cover of this central fiber is the KSB stable limit of the associated family of Horikawa surfaces.  One could similarly compute the replacement of this pair in $\calM_8^{1/2 +\epsilon}$ after the $K_Y + (\frac{1}{2} + \epsilon)\Delta$ flip of $\widetilde{C}_0$.  
\end{example}

In both of these cases, the smoothness of the replacement curve is tied to the rationality of the unicuspidal curve.  The rationality allows the curve to be the fiber of a small contraction, flip, or flop of the total space in the first step of the minimal model problem.  As the curve only intersects the exceptional divisor transversally, this does not affect the smoothness of the curve $D$ in the exceptional divisor.  

Motivated by these examples, we expect that there is a relationship between non-planar limits of degree $d$ plane curves and unicuspidal rational degree $d$ plane curves where a \textit{wall crossing} as above replaces the rational unicuspidal curve with a smooth non-planar curve.  This problem is explored in more detail in \cite{DS22}, focusing on the case that $d$ is prime.  More broadly, however, to effectively enumerate the walls $\alpha_i$ in the moduli problem, one needs an understanding of the possible plane curve singularities.  Thus, for many applications, the following question is essential: 

\begin{question}\label{q:whatcurves}
    What are the rational unicuspidal plane curves of degree $d$?  
\end{question}

For any degree $d$, there always exists the rational unicuspidal plane curve $x^{d-1}z = y^d$ (Theorem \ref{thm:onepair}) as the generalization of the usual cuspidal cubic.  We are interested in other cuspidal singularities. 
Furthermore, the work in \cite{DS22} suggests that there exist very few examples of \textit{smooth} non-planar limits of smooth plane curves of \textit{prime} degree, related to the existence of prime degree rational unicuspidal plane curves.  This relationship is not sharp, but leads us to investigate the following refinement of the previous question: 

\begin{question}\label{q:whatprimecurves}
    If $d$ is prime, when do there exist rational unicuspidal plane curves of degree $d$ other than $x^{d-1}z = y^d$ (up to isomorphism)?  
\end{question}

It turns out that the answer to this question is closely intertwined with many longstanding open questions in number theory, notably the infinitude of Fibonacci primes and one of Landau's problems on primes of the form $n^2+1$.  We explore this in more detail in Section \ref{s:prime}.

In what follows, we introduce two perspectives that have been historically fruitful to yield partial answers to the problem.  In Section \ref{s:newtonpairs}, we classify curve singularities by combinatorial properties associated to the Newton pairs parameterizing the cusp, and in Section \ref{s:kodairadim}, we connect this perspective with classification of curve singularities by the logarithmic Kodaira dimension of the curve.  The main results of this article are the following. 

\begin{theorem}\label{main1}
    For degree $d \le 30$, there is explicit list of all unicuspidal rational plane curves of degree $d$, enumerated in Table \ref{t:bigtable}.
\end{theorem}

Additionally, we make two statements about the possible number of non-isomorphic curves for composite versus prime degree.

\begin{theorem}\label{main2}
    For any integer $d > 1$, there are at least as many non-isomorphic rational unicuspidal plane curves of degree $d$ as there are ordered factorizations of $d$.
\end{theorem}

\begin{theorem}\label{main3}
    There exist infinitely many prime degrees for which there exist at least two non-isomorphic rational unicuspidal plane curves.  
\end{theorem}

Conjecturally, there exist at least three different infinite families of prime degrees for which there exists more than one rational unicuspidal plane curve.  However, we expect for a general prime number $p$, there only exists (up to isomorphism) the rational unicuspidal curve $x^{p-1}z = y^p$. These conjectures and expectations will be discussed in Section \ref{s:prime}.

\vspace{1em}

\noindent \textbf{Acknowledgements.} We would like to thank David Stapleton for helpful conversations.  Research of KD was supported in part by NSF grant DMS-2302163.

\section{Classification of cuspidal curves by their Newton pairs and associated semigroup}\label{s:newtonpairs}

We recall the classical invariants parametrizing a cusp. For additional background references see \cite[Ch. 8]{PlaneAlgCurves}, \cite[Ch. 2]{Liu2014}, and \cite[Ch. 2]{MoeThesis}.

Let $p \in C \subset  X$ be a cuspidal point $p$ on the curve $C$ on a smooth surface $X$.  In an analytic neighborhood of $p$, the curve $C$ can be written parametrically as
\[
(x,y) = (t^a, c_1t^{b_1} + c_2t^{b_2} + \dots)
\]
with $1 < a < b_1 < b_2 < \dots\in \bZ$ such that $a$ does not divide $b_1$, $\gcd(a, b_1, b_2, \dots) = 1$, and $c_i \ne 0$ for each $i$. This parametrization determines the geometry of the curve.  We begin by recalling several preliminary definitions.

\begin{definition}\label{defn:delta}
Let $C$ be a reduced curve with normalization $\nu:C^\nu \rightarrow C$ and let $p\in C$ be a point.  The \textbf{$\delta$-invariant} of $C$ at $p$ is the number
\[
\delta_p=\mathrm{length}_p\left(\nu_*(\calO_{C^\nu})/\calO_C\right).
\]
\end{definition}

The sum of the $\delta$-invariants controls the difference between the arithmetic and geometric genera of $C$:
\[
\chi(\calO_{C^\nu})-\chi(\calO_C) = \sum_{p\in C} \delta_p.
\]

\begin{remark}\label{rmk:rationaldelta}
    If $C$ is a rational unicuspidal plane curve of degree $d$ with cusp at $p \in C$, then we must have the $\delta_p$ is actually equal to the arithmetic genus because $g(\calO_{C^\nu}) = 0$.  Therefore, we have the equality 

\[ \delta_p = \frac{(d-1)(d-2)}{2}. \]
\end{remark}

\begin{definition}\label{defn:multiplicitysequence}
The \textbf{multiplicity sequence} of a cusp is the sequence encoding the multiplicity of the exceptional divisors in the minimal resolution of the cusp.  Let $(X_0, C_0) := (X,C)$ and let
\[
\pi_i: (X_i,C_i) \to (X_{i-1}, C_{i-1})
\]
be the blow up of the singular point of $C_{i-1}$ with exceptional divisor $E_i$. Set $C_i = (\pi_i)^{-1}_* C_{i-1}$.  Let $\pi = \pi_n \circ \pi_{n-1} \dots \circ \pi_1$ be the minimal embedded resolution of the cusp $p$.  The \textit{multiplicity sequence of the cusp} $p \in C$ is the sequence $\overline{m}_p := (m_1, m_2, \dots, m_n)$, where $m_i$ is the multiplicity of the exceptional divisor $E_i$ in $\pi_i^*(C_{i-1})$.  This satisfies the inequalities $m_1 \ge m_2 \ge \dots \ge m_n = 1$.  For simplicity, as $C_{i-1}$ is smooth if and only if $m_i = 1$ (and hence $C_j$ smooth for all $i \le j \le n$), we omit all the multiplicities $m_j =1$ for $j\ge i$.  

To further simplify the notation, we will write multiplicity sequences as 
\[\overline{m}_p = (m_1, m_2, \dots, m_n) = ((k_1)_{i_1}, (k_2)_{i_2}, \dots, (k_n)_{i_n})\] where we collect repeated entries and the subscript indicates the number of times each entry is repeated.  For instance, instead of writing $(4,2,2,2,2)$, we will write $(4,2_4)$.  
\end{definition}

The $\delta$-invariant (Definition \ref{defn:delta}) of a cusp singularity can be read off from its multiplicity sequence: 
\begin{equation}\label{deltamultiplicity}
    \delta_p = \sum_{i = 0}^n \frac{m_i(m_i-1)}{2} 
\end{equation}

\begin{definition}\label{defn:newtonpairs}
Another invariant of the cusp is the collection of \textbf{Newton pairs} that parameterize the cusp.  Define $g_i := \gcd(a,b_1,b_2, \dots ,b_i)$.  Then, there is a finite sequence $i_1 < i_2 < \dots < i_k$ at which $g_i$ decreases, i.e. $i_1 = 1$,
\[
g_{i_1} = \dots = g_{i_2 - 1} > g_{i_2} = \dots =g_{i_3-1} > g_{i_3} = \dots > g_{i_k} = 1.
\]
The monomials $t^{b_{i_j}}$ appearing in the parameterization are called the \textbf{characteristic terms}.

Define $i_0 = 0$, $b_0 = 0$, and $g_{i_0} = a$. For $1 \le j \le k+1$, let $P_j = g_{i_{j-1}}$ and for $1 \le j \le k$, let $Q_j = b_{i_j} - b_{i_{j-1}}$.  The $k$ \textbf{Newton pairs} of the cusp are the $k$ pairs
\[
(p_j, q_j) = \left(\frac{P_j}{P_{j+1}}, \frac{Q_j}{P_{j+1}} \right)\text{ for } 1 \le j \le k .
\]
\end{definition}

We can write $P_j = p_j p_{j+1} \dots p_k$ and $Q_j = q_j p_{j+1} \dots p_k$. The $\delta$-invariant of a cusp singularity can also be expressed in terms of the $P_i$ \cite[2.1.1]{Liu2014}:
\begin{equation}\label{delta}
\delta_p = \frac{1}{2} \left((P_1-1)(Q_1-1) + \sum_{j=2}^k(P_j-1)Q_j \right).
\end{equation}

\begin{remark}
    The pairs $(P_1, Q_1), (P_2, Q_2), \dots , (P_k,Q_k)$ are called the \textbf{Puiseux pairs} of the cusp.  These can be read off from the Newton pairs, and vice versa.  The multiplicity sequence can also be determined from the Newton pairs using an extended Euclidean algorithm as in \cite[Theorem 2.2.6]{MoeThesis}, and vice versa. 
\end{remark}

\begin{definition}\label{defn:semigroup}
The \textbf{semigroup} associated to a cuspidal singularity $p\in C$, denoted $W_p\subset \mathbb{N}$, is the set of all possible local intersection multiplicities of $C$ with other curves at $p$:
\[
W_p:= \{ \dim_{\mathbb{C}} \calO_{C,p} /(f) \mid f \in \mathcal{O}_{C,p}, f \ne 0 \} \subset \mathbb{N}.
\]
\end{definition}

The semigroup $W_p$ has a set of minimal generators $\{0,w_1, w_2, \dots , w_{k+1} \}$, and each $w_i$ can be expressed in terms of the Newton and Puiseux pairs of the cusp:
\[
w_1 = P_1, \quad w_2 = Q_1, \quad w_j = p_{j-2}w_{j-1} + Q_{j-1} \quad 3 \le j \le k+1.
\]

\begin{example}
In the case of a cusp $p\in C$ with one Newton pair $(a,b)$, the curve can be analytically locally parametrized by $(x,y)=(t^a,t^b)$ with $\gcd(a,b)=1$. In this case, $C$ has analytic local equation $y^a=x^b$, $\delta$-invariant $(a-1)(b-1)/2$, and semigroup $W_p=\langle 0,a,b\rangle.$
\end{example}

Finally, if $p\in C$ is a cusp with semigroup $W_p$, define
\[
R_p(k) := \#\{ W_p \cap [0,k) \}
\]
to be the counting function of elements in $W_p$ between $0$ and $k-1$. It has been shown that if a curve $C$ is a cuspidal degree $d$ plane curve, these counting functions must satisfy particular constraints related to the degree $d$ of the curve.  The multiplicity, Newton pairs, and intervals in the semigroup have been widely used to study plane curves, e.g. \cite{MatsuokaSakai,Orevkov,FdBLMHN2,BorodzikLivingston}.

In \cite[Thm. 6.5, Rem. 6.6]{BorodzikLivingston}, Borodzik and Livingston prove the following strong result on existence of rational cuspidal curves: 

\begin{theorem}{\cite{BorodzikLivingston}}
    For a rational cuspidal curve of degree $d$ with $n$ cusps $p_1, \dots, p_n$ and counting functions $R_{p_1}, \dots, R_{p_n}$, then for any $j \in \{ -1, \dots, d-2\}$, 

\begin{equation}\label{countingfunction}
    \min_{\substack{k_1, \dots, k_n \in \mathbb{Z}; \\ k_1 + \dots + k_n = jd + 1}} ( R_{p_1}(k_1) + \dots + R_{p_n}(k_n)) = \frac{(j+1)(j+2)}{2}.
\end{equation}
\end{theorem}

With this background, we will classify low degree curves by the number of Newton pairs parameterizing the cusp using Equation \ref{countingfunction}.  Currently, there exists the following classification of unicuspidal curves with cusp parameterized by one or two Newton pairs.  We denote by $\phi_j$ the Fibonacci numbers, defined by $\phi_0 =0$, $\phi_1 =1$, and $\phi_{j+1} = \phi_j + \phi_{j-1}$.  

\begin{theorem}{\cite{FdBLVMHN1,FdBLMHN2}}\label{thm:onepair}
    A rational unicuspidal plane curve of degree $d$ with singularity parametrized by one Newton pair $(p_1,q_1)$ exists if and only if $(d,p_1,q_1)$ is in the following list. 
        \begin{enumerate}
            \item $d$ is arbitrary and $(p_1,q_1) = (d-1, d)$
            \item $d$ is even and $(p_1,q_1) = (d/2, 2d-1)$
            \item $d = \phi_{j-1}^2+1 = \phi_{j-2}\phi_j$ where $j \ge 5$ is odd and $(p_1,q_1) = (\phi_{j-2}^2, \phi_j^2) $
            \item $d = \phi_{j}$ where $j \ge 5$ is odd and $(p_1,q_1) = (\phi_{j-2}, \phi_{j+2}) $
            \item $d = 8$ and $(p_1,q_1) = (3,22)$
            \item $d = 16$ and $(p_1,q_1) = (6,43)$.
        \end{enumerate}
\end{theorem}

\begin{theorem}{\cite{Liu2014,Bodnar2016}}\label{thm:twopairs}
    A rational unicuspidal plane curve of degree $d$ with singularity parametrized by two Newton pairs $(p_1,q_1), (p_2,q_2)$, exists if and only if $(d, p_i,q_i)$ appear in the following list. 
        \begin{enumerate}
            \item $d = \phi_{2k-1} \phi_{2k+1}(l\phi_{2k-1}^2+\phi_{2k-3}^2)$ for some $k \ge 2$ and $l \ge 0$ but $(k,l) \ne (2,0)$, and 
            
            $(p_1,q_1),(p_2,q_2) = (l\phi_{2k-1}^2 + \phi_{2k-3}^2, l\phi_{2k+1}^2 + \phi_{2k-1}^2 + 2),(\phi_{2k-1}^2,l\phi_{2k-1}^2 + \phi_{2k-3}^2) $ 
            \item $d =  \phi_{2k+1}(l\phi_{2k-1}^2+\phi_{2k-3}^2)$ for some $k \ge 3$ and $l \ge 0$, and 
            
            $(p_1,q_1),(p_2,q_2) = (l\phi_{2k-1}^2 + \phi_{2k-3}^2, l\phi_{2k+1}^2 + \phi_{2k-1}^2 + 2),(\phi_{2k-1},l\phi_{2k-1} + \phi_{2k-5}) $
            \item $d = nm$ for some $n \ge 3, m\ge 2$, and $(p_1,q_1),(p_2,q_2) = (n-1,n),(m,nm-1)$
            \item $d = 2nm$ for some $n \ge 2, m\ge 2$, and $(p_1,q_1),(p_2,q_2) = (n,4n-1),(m,nm-1)$
            \item $d = n^2+1$ for some $n \ge 3$, and $(p_1,q_1),(p_2,q_2) = (n-1,n),(n,(n+1)^2)$
            \item $d = 8n^2+4n+1$ for some $n \ge 2$, and $(p_1,q_1),(p_2,q_2) = (n,4n+1),(4n+1,(2n+1)^2)$
            \item $d = \phi_{4k+2}$ for some $k \ge 2$, $(p_1,q_1),(p_2,q_2) = (\phi_{4k}/3, \phi_{4k+4}/3),(3,1)$
            \item $d = 2\phi_{4k+2}$ for some $k \ge 2$, $(p_1,q_1),(p_2,q_2) = (\phi_{4k}/3, \phi_{4k+4}/3),(6,1)$
        \end{enumerate}
\end{theorem}

Beyond the case of one or two Newton pairs, very little is known.  The first main result of this article is the following. 

\begin{theorem}\label{thm:threepairs}
A unicuspidal rational curve of degree $ d \le 30$, with cusp parametrization given by 3 Newton pairs, is one of the curves in Table \ref{t:threepairs}. 
\begin{table}[h]
\caption{Rational unicuspidal curves with cusp parameterized by three Newton pairs.}
\label{t:threepairs}
\begin{tabular}{|c|c|c|c|}
    \hline degree & Newton pairs & sample parametrization\tablefootnote{This parametrization only includes the characteristic terms in the local equation of the cusp, which are the data contained in the Newton pairs.  A cusp with parametrization $(x,y) = (t^a, t^{b_{i_1}} + t^{b_{i_2}}+ \dots )$ where $g_{i_j} > g_{i_{j+1}}$ (notation as in Definition \ref{defn:newtonpairs}), has the same multiplicity sequence and collection of Newton pairs as $(x,y) = (t^a, u_{i_1}t^{b_{i_1}} + u_{i_2}t^{b_{i_2}}+ \dots )$ where each $u_{i_j}$ is a unit in the power series ring such that $\deg u_{i_j}t^{b_{i_j}} < b_{i_{j+1}}$ and for any monomial $t^{b}$ appearing in $u_{i_j}t^{b_{i_j}}$, $\gcd(a,b_{i_j},b) = \gcd(a,b_{i_j})$.\label{Footnote}} & multiplicity sequence  \\
    \hline 12 & $(2, 3), (2, 5), (2, 3)$ & $(x,y) = (t^8, t^{12} + t^{22} + t^{25})$ & $(8, 4_4, 2_3)$  \\
    \hline 16 & $(2, 7), (2, 3), (2, 3)$ & $(x,y) = (t^8, t^{28} + t^{34} + t^{37})$ & $(8_3, 4_3, 2_3)$  \\
    \hline 16 & $(3, 4), (2, 7), (2, 3)$ & $(x,y) = (t^{12}, t^{16} + t^{30} + t^{33})$ & $(12, 4_6, 2_3)$  \\
    \hline 18 & $(2, 3), (2, 5), (3, 5)$ & $(x,y) = (t^{12}, t^{18} + t^{33} + t^{38})$ & $(12, 6_4, 3_3, 2)$  \\
    \hline 18 & $(2, 3), (3, 8), (2, 5)$ & $(x,y) = (t^{12}, t^{18} + t^{34} + t^{39})$ & $(12, 6_4, 4, 2_4)$  \\
    \hline 19 & $(2, 3), (2, 7), (3, 7)$ & $(x,y) = (t^{12}, t^{18} + t^{39} + t^{46})$ & $(12, 6_5, 3_4)$  \\
    \hline 20 & $(4, 5), (2, 9), (2, 3)$ & $(x,y) = (t^{16}, t^{20} + t^{38} + t^{41})$ & $(16, 4_8, 2_3)$  \\ 
    \hline 24 & $(2, 7), (2, 3), (3, 5)$ & $(x,y) = (t^{12}, t^{42} + t^{51} + t^{56})$ & $(12_3, 6_3, 3_3, 2)$ \\
    \hline 24 & $(2, 7), (3, 5), (2, 5)$ &  $(x,y) = (t^{12}, t^{42} + t^{52} + t^{57})$ & $(12_3, 6_3, 4, 2_4)$  \\
    \hline 24 & $(3, 11), (2, 5), (2, 3)$ &  $(x,y) = (t^{12}, t^{44} + t^{54} + t^{57})$ & $(12_3, 8, 4_4, 2_3)$  \\
    \hline 24 & $(2, 3), (2, 5), (4, 7)$ & $(x,y) = (t^{16}, t^{24} + t^{44} + t^{51})$& $(16, 8_4, 4_3, 3)$  \\
    \hline 24 & $(2, 3), (4, 11), (2, 7)$ & $(x,y) = (t^{16}, t^{24} + t^{46} + t^{53})$ & $(16, 8_4, 6, 2_6)$ \\
    \hline 24 & $(3, 4), (2, 7), (3, 5)$ & $(x,y) = (t^{18}, t^{24} + t^{45} + t^{50})$ & $(18, 6_6, 3_3, 2)$  \\
    \hline 24 & $(3, 4), (3, 11), (2, 5)$ & $(x,y) = (t^{18}, t^{24} + t^{46} + t^{51})$ & $(18, 6_6, 4, 2_4)$ \\
    \hline 24 & $(5, 6), (2, 11), (2, 3)$ & $(x,y) = (t^{20}, t^{24} + t^{46} + t^{49})$ & $(20, 4_{10}, 2_3)$ \\
    \hline 27 & $(2, 3), (3, 8), (3, 8)$ & $(x,y) = (t^{18}, t^{27} + t^{51} + t^{59})$ & $(18, 9_4, 6, 3_4, 2)$  \\
    \hline 28 & $(2, 3), (3, 10), (3, 10)$ & $(x,y) = (t^{18}, t^{27} + t^{57} + t^{67})$ & $(18, 9_5, 3_6)$  \\
    \hline 28 & $(6, 7), (2, 13), (2, 3)$ & $(x,y) = (t^{24}, t^{28} + t^{54} + t^{57})$ & $(24, 4_{12}, 2_3)$  \\
    \hline 30 & $(2, 3), (2, 5), (5, 9)$ & $(x,y) = (t^{20}, t^{30} + t^{55} + t^{64})$ & $(20, 10_4, 5_3, 4)$  \\
    \hline 30 & $(2, 3), (5, 14), (2, 9)$ & $(x,y) = (t^{20}, t^{30} + t^{58} + t^{67})$ & $(20, 10_4, 8, 2_8)$  \\
    \hline 30 & $(4, 5), (2, 9), (3, 5)$ & $(x,y) = (t^{24}, t^{30} + t^{57} + t^{62})$ & $(24, 6_8, 3_3, 2)$  \\
    \hline 30 & $(4, 5), (3, 14), (2, 5)$ & $(x,y) = (t^{24}, t^{30} + t^{58} + t^{63})$ & $(24, 6_8, 4, 2_4)$  \\
    \hline
\end{tabular}
\end{table}   
\end{theorem} 

For low degree curves, we also show that there is only one curve with more than three Newton pairs. 

\begin{theorem}\label{thm:fourpairs}
    For degree $d \le 30$, the only unicuspidal rational curve of degree $d$ with cusp given by $\ge 4$ Newton pairs is the curve in Table \ref{t:fourpairs}.  
\begin{table}[h]
\caption{Rational unicuspidal curves with cusp parameterized by at least four Newton pairs.}
\label{t:fourpairs}
\begin{tabular}{|c|c|c|c|}
    \hline degree & Newton pairs & sample parametrization\tablefootnote{See \cref{Footnote}.} & multiplicity sequence  \\
    \hline 24 & $(2, 3), (2, 5), (2, 3), (2,3)$ & $(x,y) = (t^{16}, t^{24} + t^{44} + t^{50}+t^{53})$ & $(16,8_4,4_3,2_3)$  \\
    \hline
\end{tabular}
\end{table}
\end{theorem}

Combining these results, we have: 

\begin{corollary}(= Theorem \ref{main1})
    Any unicuspidal rational curve of degree $d$ for $1 \le d \le 30$ appears in Theorems \ref{thm:onepair}, \ref{thm:twopairs}, \ref{thm:threepairs}, or \ref{thm:fourpairs}.  The exhaustive list of all possible unicuspidal rational curves of degree $d \le 30$ appears in Table \ref{t:bigtable}.
\end{corollary}

The key tool to prove Theorems \ref{thm:threepairs} and \ref{thm:fourpairs} comes from low dimensional topology and the study of the distribution of the semigroup associated to the cusp as proven in \cite{BorodzikLivingston} using Equation \ref{countingfunction}.  This allows one to produce a list of candidates for unicuspidal rational curves.  From the list of potential candidates, we use blow ups and blow downs to exhibit each curve as a birational modification of a known unicuspidal rational curve of lower degree, thus proving its existence.  

\begin{proof}{(Of Theorems \ref{thm:threepairs} and \ref{thm:fourpairs}.)}
    First, note that that a bound on the number of Newton pairs of the cusp in terms of the degree of the curve is derived in \cite[Remark 5.5]{DS22}.  This bound is far from sharp, but for a curve of degree $d$ with cusp parameterized by $k$ Newton pairs, we have 
    \[ (d-1)(d-2) \ge (2^k-1)2^k. \]

    If $k = 5$, this implies that $d >30$ so beyond the scope of Theorems \ref{thm:threepairs} and \ref{thm:fourpairs}.  Thus, we only need to consider curves with cusps parameterized by three or four Newton pairs. 
    
    We use a computer program to generate a list of possible cusp candidates defined by three and four Newton pairs.  The program is included in Appendix \ref{appendix}.  We describe the program for three pairs (as the program for four is similar).

For three pairs, given a degree $d$, the program first generates a finite list of possible characteristic terms of the parameterization $a, b_1, b_2, b_3$ and uses them to generate the related Puiseux pairs $(P_1, Q_1), (P_2, Q_2), (P_3, Q_3)$ and Newton pairs $(p_1, q_1), (p_2, q_2), (p_3, q_3)$ such that:
\begin{align*}
    P_1 &= a = p_1p_2p_3\\
    P_2 &= \gcd(P_1, b_1) = p_2p_3\\
    P_3 &= \gcd(P_2, b_2) = p_3\\
    Q_1 &= b_1 = q_1p_2p_3\\
    Q_2 &= b_2 - b_1 = q_2p_3\\
    Q_3 &= b_3 - b_2 = q_3
\end{align*}
where the maximum of $a, b_i$ is bounded by the genus of the curve.

The program then checks a series of properties:
\begin{enumerate}
    \item Validity for each Newton pair: $\gcd(p_i, q_i) = 1$, and that the sequence $\{(P_i, Q_i)\}$ satisfies the conditions in Definition \ref{defn:newtonpairs}.
    
    \item The $\delta$-invariant is equal to the genus of a smooth degree $d$ curve (so the normalization is rational): $\frac{(d-1)(d-2)}{2} = \frac{1}{2}\left( (P_1-1)(Q_1-1) + (P_2-1)(Q_2) + (P_3-1)(Q_3) \right)$.
    
    \item The semigroup counting property \ref{countingfunction} is satisfied: the generators for relevant semigroup $W = \langle 0, w_1, w_2, w_3, w_4 \rangle$ are computed by
    \begin{align*}
        w_1 &= P_1\\
        w_2 &= Q_1\\
        w_3 &= p_1 w_2 + Q_2\\
        w_4 &= p_2 w_3 + Q_3
    \end{align*}
    Since there is only one cusp, the semigroup counting property says for any $l$ up to $d - 2$, the size of $W \cap [0, ld)$ should be equal to $\frac{(l+1)(l+2)}{2}$.
\end{enumerate}

After these checks, we are left with a list of possible curves of degree $d$ defined by three Newton pairs.  The program in fact yields the list in Theorem \ref{thm:threepairs} and Theorem \ref{thm:fourpairs}.  However, the checks and semigroup counting property \ref{countingfunction} do not guarantee the existence of such a curve.  So, to complete the proof, we must verify the existence of each cusp.  We do this by explicit construction with the following two lemmas.

\begin{lemma}\label{lem:existenceofmostcurves}
    Let $k$ and $n$ be positive integers, $n \ge 2$.  Suppose $C$ is a unicuspidal plane curve of degree $d = (k+1)n$ with multiplicity sequence of the cusp $\overline{m} = (kn, n_{2k}, \overline{m}')$, where $\overline{m'}$ indicates the rest of the sequence.  Then, $C$ can be transformed by blow ups and blow downs to a unicuspidal curve $C'$ of degree $n$ with multiplicity sequence $\overline{m}'$.  In particular, as these monoidal transformations are invertible, $C$ exists if and only if $C'$ exists.
\end{lemma}

\begin{proof}
    Suppose $C$ is a plane curve of degree $d = (k+1)n$ with a cusp $p \in C$ with mulitplicity sequence $\overline{m} = (kn, n_{2k}, \overline{m}')$.  For $T$ the tangent line to $C$ at $p$, we have $m_1+m_2 \le (T\cdot C)_p \le d$ by \cite[(2.3)]{MoeThesis}, so $(k+1)n \le (T\cdot C)_p \le (k+1)n$, so $(T\cdot C)_p = T\cdot C = (k+1)n = d$.  

    Repeatedly blow up the intersection point of $C$ and $T$, and let $E_i$ be the exceptional divisor of the $i$th blow up.  After $k+1$ blow ups, denote by $\widehat{C}$ and $\widehat{T}$ the strict transform of $C$ and $T$.  The dual graph of the configuration of curves is given below, where we adopt the convention that we use \textit{straight edges} to indicate transverse intersections (as one would in a typical dual graph) but \textit{curved edges} to indicate higher multiplicity intersections: 

    
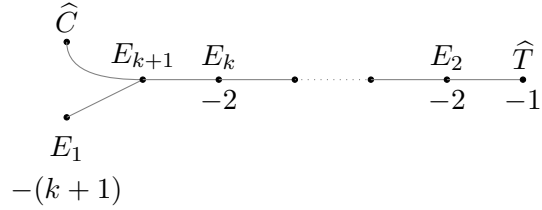
\begin{center}
        \begin{tikzpicture}
            \filldraw[black] (-1,0.5) circle (1pt) node[above]{$\widehat{C}$};
            \draw[gray] (-1,0.5)  to [out=-90,in=180] (0,0);
            \filldraw[black] (-1,-0.5) circle (1pt) node[below]{ $\begin{array}{c} E_1 \\ -(k+1) \end{array}$ };
            \draw[gray] (-1,-0.5) -- (0,0);
            \filldraw[black] (0,0) circle (1pt) node[above]{$E_{k+1}$};
            \draw[gray] (0,0) -- (1,0);
            \filldraw[black] (1,0) circle (1pt) node[above]{$E_{k}$} node[below]{$-2$};
            \draw[gray] (1,0) -- (2,0);
            \filldraw[black] (2,0) circle (1pt);
            \draw[gray, dotted] (2,0) -- (3,0);
            \filldraw[black] (3,0) circle (1pt);
            \draw[gray] (3,0) -- (4,0);
            \filldraw[black] (4,0) circle (1pt) node[above]{$E_{2}$} node[below]{$-2$};
            \draw[gray] (4,0) -- (5,0);
            \filldraw[black] (5,0) circle (1pt) node[above]{$\widehat{T}$} node[below]{$-1$};
        \end{tikzpicture}
    \end{center}

    Here, $\widehat{C}$ has multiplicity sequence $(n_k, \overline{m}')$ and $\widehat{C} \cdot E_{k+1} = n$.  Here, the intersections of $E_i$ and $T$ are transverse, but the intersection of $\widehat{C}$ and $E_{k+1}$ is the singular point of $\widehat{C}$.  Contracting the chain $E_k \cup \dots \cup E_2 \cup \widehat{T}$ yields the Hirzebruch surface $\bF_{k+1}$ with negative section $E_1$ and fiber $E_{k+1}$.  The curve $\widehat{C}$ does not intersect the negative section.  

    On the Hirzebruch surface, blow up the cusp of $\widehat{C}$ (the intersection of $\widehat{C}$ and $E_{k+1}$) $k$ times. Denote the exceptional divisor of the $i$th blow up by $F_i$ and the strict transform of $\widehat{C}$, $E_{k+1}$, and $E_1$ by $C'$, $E_{k+1}'$, and $E_1'$, respectively.  From the multiplicity sequence of $\widehat{C}$, the configuration of curves has dual graph: 

    
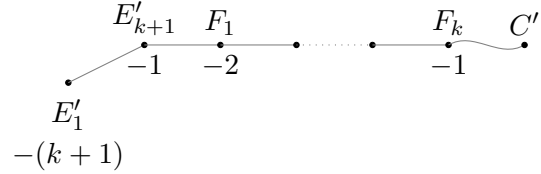
\begin{center}
        \begin{tikzpicture}
            \filldraw[black] (-1,-0.5) circle (1pt) node[below]{ $\begin{array}{c} E_1' \\ -(k+1) \end{array}$ };
            \draw[gray] (-1,-0.5) -- (0,0);
            \filldraw[black] (0,0) circle (1pt) node[above]{$E_{k+1}'$} node[below]{$-1$};
            \draw[gray] (0,0) -- (1,0);
            \filldraw[black] (1,0) circle (1pt) node[above]{$F_{1}$} node[below]{$-2$};
            \draw[gray] (1,0) -- (2,0);
            \filldraw[black] (2,0) circle (1pt);
            \draw[gray, dotted] (2,0) -- (3,0);
            \filldraw[black] (3,0) circle (1pt);
            \draw[gray] (3,0) -- (4,0);
            \filldraw[black] (4,0) circle (1pt) node[above]{$F_{k}$} node[below]{$-1$};
            \draw[gray] (4,0) to [out=30,in=210] (5,0);
            \filldraw[black] (5,0) circle (1pt) node[above]{$C'$};
        \end{tikzpicture}
    \end{center}

    Here, $C'\cdot F_{k} = n$ (meeting at the cusp of $C'$) and $C'$ has multiplicity sequence $\overline{m}'$.  We contract $E_1' \cup E_{k+1}' \cup F_1 \cup \dots \cup F_{k-1}$ to a smooth point in $\bP^2$.  The image of $F_k$ is a line in $\bP^2$ and as $C'\cdot F_k = n$, $C'$ is a degree $n$ unicuspidal plane curve with multiplicity sequence $\overline{m}'$. 
\end{proof}

Because $C'$ has lower degree than $C$, this provides an inductive approach to proving the existence of $C$.  From this lemma, we  can verify the existence of all but two curves in Table \ref{t:threepairs}.  In Table \ref{t:induct}, we use Lemma \ref{lem:existenceofmostcurves} to reduce the existence of each curve $C$ to a lower degree curve $C'$, and in some cases reduce $C'$ to a lower degree curve $C''$.  

\begin{center}
\begin{table}[h]
\caption{Proving existence of curves by reducing their degree.}
\label{t:induct}
\begin{tabular}{|c|p{.15\textwidth}|p{.3\textwidth}|p{.3\textwidth}|}
    \hline degree & multiplicity sequence of $C$ & existence of $C$ is equivalent to existence of $C'$ with $d'$, $\overline{m}'$ & existence of $C'$ is equivalent to existence of $C''$ with $d''$, $\overline{m}''$ \\
    \hline 12 & $(8, 4_4, 2_3)$ & $d' = 4, \overline{m}' = (2_3)$ &  \\
    \hline 16 & $(8_3, 4_3, 2_3)$ & $d' = 8, \overline{m}' = (4_3, 2_3)$ & $d'' = 4, \overline{m}'' = (2_3)$ \\
    \hline  &  $(12, 4_6, 2_3)$  & $d' = 4, \overline{m}' = (2_3) $ & \\
    \hline 18 & $(12, 6_4, 3_3, 2)$ & $d' = 6, \overline{m}' = (3_3, 2)$ & \\
    \hline  &  $(12, 6_4, 4, 2_4)$ & $d' = 6, \overline{m}' = (4, 2_4)$ &  \\
    \hline 20 & $(16, 4_8, 2_3)$ & $d' = 4, \overline{m}' = (2_3)$ & \\
    \hline 24 & $(12_3, 6_3, 3_3, 2)$ & $d' = 12, \overline{m}' = (6_3, 3_3, 2)$ & $d'' = 6, \overline{m}'' = (3_3, 2)$ \\
    \hline 24 & $(12_3, 6_3, 4, 2_4)$ & $d' = 12, \overline{m}' = (6_3, 4, 2_4)$ & $d'' = 6, \overline{m}'' = (4, 2_4)$   \\
    \hline 24 & $(12_3, 8, 4_4, 2_3)$ & $d' = 12, \overline{m}' = (8, 4_4, 2_3)$ & $d'' = 4, \overline{m}'' = (2_3)$  \\
    \hline 24 & $(16, 8_4, 4_3, 3)$ & $d' = 8, \overline{m}' = (4_3, 3)$ & $d'' = 4, \overline{m}'' = (3)$  \\
    \hline 24 & $(16, 8_4, 6, 2_6)$ & $d' = 8, \overline{m}' = (6, 2_6)$ & $d'' = 2, \overline{m}'' = (1)$ ($C''$ is smooth)  \\
    \hline 24 & $(18, 6_6, 3_3, 2)$ & $d' = 6, \overline{m}' = (3_3, 2)$ & \\
    \hline 24 & $(18, 6_6, 4, 2_4)$ & $d' = 6, \overline{m}' = (4, 2_4)$ & \\
    \hline 24 & $(20, 4_{10}, 2_3)$ & $d' = 4, \overline{m}' = (2_3)$ &  \\
    \hline 27 & $(18, 9_4, 6, 3_4, 2)$ & $d' = 9, \overline{m}' = (6, 3_4, 2)$ & $d'' = 3, \overline{m}'' = (2)$ \\
    \hline 28 & $(24, 4_{12}, 2_3)$  & $d' = 4, \overline{m}' = (2_3)$ & \\
    \hline 30 & $(20, 10_4, 5_3, 4)$ & $d' = 10, \overline{m}' = (5_3, 4)$ & $d'' = 5, \overline{m}'' = (4)$ \\
    \hline 30 &  $(20, 10_4, 8, 2_8)$ & $d' = 10, \overline{m}' = (8, 2_8)$ & $d'' = 2, \overline{m}'' = (1)$ ($C''$ is smooth)   \\
    \hline 30 & $(24, 6_8, 3_3, 2)$ & $d' = 6, \overline{m}' = (3_3, 2)$ & \\
    \hline 30 & $(24, 6_8, 4, 2_4)$ & $d' = 6, \overline{m}' = (4, 2_4)$ &  \\
    \hline
\end{tabular}
\end{table}  
\end{center}

For the only curve in Table \ref{t:fourpairs}, we can apply Lemma \ref{lem:existenceofmostcurves} twice to show that the existence of $C$ is equivalent to the existence of a degree $8$ curve $C'$ with multiplicity sequence of the cusp $\overline{m}' =(4_3, 2_3)$, which is equivalent to the existence of a degree $4$ curve $C''$ with $\overline{m}'' = (2_3)$. 

In summary, we can reduce the existence of every curve to either a smooth conic or a degree 4, 5, or 6 unicuspidal rational curve, and then appeal to \cite[Table 1]{DS22} to verify their existence.  Indeed, there exist curves of degree 4 with multiplicity sequences $(2_3)$ and $(3)$; there exists a curve of degree 5 with multiplicity sequence $(4)$; and there exist curves of degree 6 with multiplicity sequences $(3_3, 2)$ and $(4, 2_4)$ which are explicitly given in \cite[Table 1]{DS22}.

The only curves in Table \ref{t:threepairs} in Theorem \ref{thm:threepairs} that are not present in Table \ref{t:induct} are the degree 19 curve and the first degree 28 curve.  We use the following lemma to verify the existence of these curves.  In the notation below, if $a = 3$ and $s = 2$ (respectively, $3$) this recovers the degree 19 curve (respectively, the first degree 28 curve).

\begin{lemma}\label{lem:type1curves}
    Let $C$ be the degree $a+1$ plane curve $zx^a = y^{a+1}$, which has a cuspidal singularity as in Theorem \ref{thm:onepair}.1 at $[0:0:1]$.  By a prescribed sequence of blow ups and blow downs, $C$ can be transformed into a rational unicuspidal plane curve $C'$ of degree $d = a^2s+1$ with multiplicity sequence $((a-1)as, as_{2a-1}, a_{2s})$. In particular, because $C$ exists, $C'$ also exists. 
\end{lemma}

\begin{proof}
    $C$ has multiplicity sequence $(a)$ at the point $p = [0:0:1]$.  Let $T$ be the inflectional line at the point $q = [1:0:0]$ given by $z = 0$.  This meets the curve to order $a+1$ at $q$.  Let $L$ be the line  $y = 0$ connecting $p$ and $q$, so $L$ meets $C$ to order $a$ at $p$ and transversely at $q$.  

    Blow up $q$ and denote the exceptional divisor by $\Sigma$, and blow up the intersection of $\Sigma$ and the strict transform of $L$ $s-1$ times, denoting the exceptional divisor of the $i$th blow up by $E_i$ to obtain the dual graph:

    
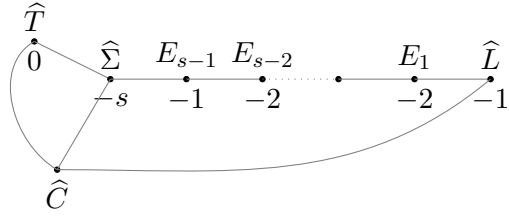
\begin{center}
        \begin{tikzpicture}
            \filldraw[black] (-1,0.5) circle (1pt) node[above]{$\widehat{T}$} node[below]{$0$};
            \draw[gray] (-1,0.5) -- (0,0);
            \filldraw[black] (0,0) circle (1pt) node[above]{$\widehat{\Sigma}$} node[below]{$-s$};
            \draw[gray] (0,0) -- (1,0);
            \filldraw[black] (1,0) circle (1pt) node[above]{${E}_{s-1}$} node[below]{$-1$};
            \draw[gray] (1,0) -- (2,0);
            \filldraw[black] (2,0) circle (1pt) node[above]{$E_{s-2}$} node[below]{$-2$};
            \draw[gray, dotted] (2,0) -- (3,0);
            \filldraw[black] (3,0) circle (1pt);
            \draw[gray] (3,0) -- (4,0);
            \filldraw[black] (4,0) circle (1pt) node[above]{$E_{1}$} node[below]{$-2$};
            \draw[gray] (4,0) -- (5,0);
            \filldraw[black] (5,0) circle (1pt) node[above]{$\widehat{L}$} node[below]{$-1$};
            \filldraw[black] (-.7,-1.2) circle (1pt) node[below]{$\widehat{C}$};
            \draw[gray] (-.7,-1.2) to [out=150,in=-150] (-1,0.5);
            \draw[gray] (-.7,-1.2) to [out=0,in=220] (5,0);
            \draw[gray] (-.7,-1.2) -- (0,0);
        \end{tikzpicture}
    \end{center}

    Here, $\widehat{C}$, $\widehat{T}$, and $\widehat{L}$ denote the strict transforms, and the intersection of every pair of curves not including $\widehat{C}$ is transverse.  The intersection of $\widehat{C}$ and $\widehat{T}$ occurs at a smooth point of $\widehat{C}$ and $\widehat{T}$ and they meet to order $a$.  This is also the (transverse) intersection point of $\widehat{C}$ and $\widehat{\Sigma}$.  The intersection of $\widehat{C}$ and $\widehat{L}$ is the cusp with multiplicity sequence $(a)$.  Now, contract the chain $\widehat{L}\cup E_1 \cup \dots \cup E_{s-2}$ to yield the Hirzebruch surface $\bF_s$.  Let $\overline{C}$ be the image of $\widehat{C}$ and $\overline{E}_{s-1}$ the image of $E_{s-1}$.  The curve $\overline{C}$ has a cusp with multiplicity sequence $(a_s)$ where it meets $\overline{E}_{s-1}$.  

    Now, blow up a point on $\overline{E}_{s-1}$ away from the cusp of $\overline{C}$ and away from the intersection with the negative section of $\bF_s$, denoting the exceptional divisor by $F_1$.  Blow up a point on $F_1$ away from the intersection with $\overline{E}_{s-1}$ and denote the exceptional divisor by $F_2$.  Repeat until we have a collection of $s$ exceptional divisors with dual graph:

    
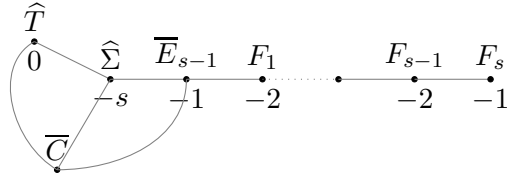
\begin{center}
        \begin{tikzpicture}
            \filldraw[black] (-1,0.5) circle (1pt) node[above]{$\widehat{T}$} node[below]{$0$};
            \draw[gray] (-1,0.5) -- (0,0);
            \filldraw[black] (0,0) circle (1pt) node[above]{$\widehat{\Sigma}$} node[below]{$-s$};
            \draw[gray] (0,0) -- (1,0);
            \filldraw[black] (1,0) circle (1pt) node[above]{$\overline{E}_{s-1}$} node[below]{$-1$};
            \draw[gray] (1,0) -- (2,0);
            \filldraw[black] (2,0) circle (1pt) node[above]{$F_1$} node[below]{$-2$};
            \draw[gray, dotted] (2,0) -- (3,0);
            \filldraw[black] (3,0) circle (1pt);
            \draw[gray] (3,0) -- (4,0);
            \filldraw[black] (4,0) circle (1pt) node[above]{$F_{s-1}$} node[below]{$-2$};
            \draw[gray] (4,0) -- (5,0);
            \filldraw[black] (5,0) circle (1pt) node[above]{$F_s$} node[below]{$-1$};
            \filldraw[black] (-.7,-1.2) circle (1pt) node[above]{$\overline{C}$};
            \draw[gray] (-.7,-1.2) to [out=150,in=-150] (-1,0.5);
            \draw[gray] (-.7,-1.2) to [out=0,in=270] (1,0);
            \draw[gray] (-.7,-1.2) -- (0,0);
        \end{tikzpicture}
    \end{center}

    and now contract $\overline{E}_{s-1} \cup F_1 \cup \dots \cup F_s$ to $\bP^1 \times \bP^1$ where the images of $\widehat{T}$ and $\widehat{\Sigma}$ are the vertical and horizontal sections, respectively, and the image of $\overline{C}$ meets $\widehat{T}$ to order $a$ at a smooth point of both curves, and meets $\widehat{\Sigma}$ at two points: transversely at the intersection of $\widehat{T}$ and $\widehat{\Sigma}$, and has a cusp with multiplicity sequence $(a_{2s})$ at the other intersection point. 

    Next, blow up the point $t \in \widehat{T}$ that is intersection of $\widehat{T}$ and $\widehat{\Sigma}$ $a$ times to separate $\widehat{T}$ and $\overline{C}$, and contract the union of all but the last exceptional divisor of this blow up and the strict transform of $\widehat{\Sigma}$. This produces the Hirzebruch surface $\bF_a$, with negative section the image of $\widehat{T}$, and the image of the last exceptional divisor is a fiber $G_0$ of this surface.  The image of $C$, $\widetilde{C}$, is a unicuspidal curve with multiplicity sequence $(as_a, a_{2s})$ and meets the fiber $G_0$ in one smooth point $g_0$ at the cusp of $\widetilde{C}$.  

    Finally, blow up $g_0$ on $G_0$, denoting the exceptional divisor by $H_1$.  Let $h_1$ be the intersection of the strict transform of $\widetilde{C}$ and $H_1$. Now, blow up $h_1$ on $H_1$, with exceptional divisor $H_2$, and iterate until there exist $a-1$ exceptional divisors:

    
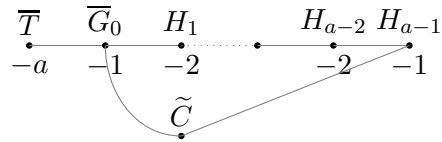
\begin{center}
        \begin{tikzpicture}
            \filldraw[black] (0,0) circle (1pt) node[above]{$\overline{T}$} node[below]{$-a$};
            \draw[gray] (0,0) -- (1,0);
            \filldraw[black] (1,0) circle (1pt) node[above]{$\overline{G}_{0}$} node[below]{$-1$};
            \draw[gray] (1,0) -- (2,0);
            \filldraw[black] (2,0) circle (1pt) node[above]{$H_1$} node[below]{$-2$};
            \draw[gray, dotted] (2,0) -- (3,0);
            \filldraw[black] (3,0) circle (1pt);
            \draw[gray] (3,0) -- (4,0);
            \filldraw[black] (4,0) circle (1pt) node[above]{$H_{a-2}$} node[below]{$-2$};
            \draw[gray] (4,0) -- (5,0);
            \filldraw[black] (5,0) circle (1pt) node[above]{$H_{a-1}$} node[below]{$-1$};
            \filldraw[black] (2,-1.2) circle (1pt) node[above]{$\widetilde{C}$};
            \draw[gray] (2,-1.2) to [out=180,in=270] (1,0);
            \draw[gray] (2,-1.2) -- (5,0);
        \end{tikzpicture}
    \end{center}

    Contracting the union $\overline{T} \cup \overline{G}_0 \cup H_1 \cup \dots \cup H_{a-2}$ yields the surface $\bP^2$ with a unicuspidal curve $C'$ with multiplicity sequence $((a-1)as, as_{2a-1}, a_{2s})$, as claimed.  Furthermore, the image of $H_{a-1}$ is a line in $\bP^2$, and computing the degree of $C'$ via the intersection of $\widetilde{C}$ and the pullback of this line, we find that $C'$ has degree $a^2s+1$.
\end{proof}

By Lemma \ref{lem:existenceofmostcurves} and Lemma \ref{lem:type1curves}, we have verified the existence of every curve generated by the program in Appendix \ref{appendix}, so this concludes the proof of Theorems \ref{thm:threepairs} and \ref{thm:fourpairs}.  
\end{proof}

\section{Classification of cuspidal curves by logarithmic Kodaira dimension}\label{s:kodairadim}

While the previous section yielded a classification of curves based on the $\delta$-invariant of the singular point and its Newton pairs, we will now connect that perspective with the logarithmic Kodaira dimension of the complement of the curve. 

\begin{definition}
    If $C \subset \bP^2$, the \textbf{Kodaira dimension} of open surface $\bP^2 - C$ is defined as $\kappa( \bP^2 - C) := \kappa(X,D)$ where $\pi: X \to \bP^2$ is a minimal log resolution of $C$ and $D = \Supp \pi^{-1} C$.  
\end{definition}

We summarize the current knowledge of the classification of unicuspidal rational plane curves by Kodaira dimension.  Just as in the case of classification by Newton pairs, there is no complete classification in general.  However, it is complete if we assume $\kappa(\bP^2 - C) < 2$.  For a more detailed detailed survey, see \cite{FdBLVMHN1}.

\subsection{Kodaira dimension $-\infty$}

 If $\kappa(\bP^2 - C) = \infty$, then $C$ is either an \textit{AMS} curve or a curve classified by Kashiwara in \cite{Kashiwara}.  AMS curves, short for Abhyankar-Moh-Suzuki, are described up to projective equivalence in \cite{Tono2000} and refer to those unicuspidal curves $C$ with cusp $p$ such that the tangent line to $C$ at $p$ meets $C$ only at $p$.  All curves in Lemma \ref{lem:existenceofmostcurves} are AMS curves. 

\subsubsection{AMS curves}\label{amscurves} 

The AMS curves are one of the following two types.  Let $d = n_1 \dots n_r$ where $r \ge 1$, $n_i$ are positive integers, where $n_i \ge 2$ for each $i$. 

\begin{enumerate}
    \item If $n_1 > 2$, the AMS curves are degree $d$ rational unicuspidal plane curves with cusp parameterized by the $r$ Newton pairs
        \[ (n_1-1, n_1), (n_2, n_1n_2-1), (n_3, n_2n_3-1), \dots, (n_k , n_{k-1}n_k-1), \dots , (n_r, n_{r-1}n_r - 1) .\]
        
    If $r = 1$ and $d = n_1$, then the cusp just has one Newton pair
    \[ (d-1, d). \]  These are precisely the curves in Theorem \ref{thm:onepair}.1. 

    For $r = 2$, these are the curves given in Theorem \ref{thm:twopairs}.3.

    \item If $n_1 = 2$ (and necessarily $r \ge 2$), so $d = 2n_2 \dots n_r$, the AMS curves are degree $d$ rational unicuspidal plane curves with cusp parametrized by the $r -1$ Newton pairs 
\[ (n_2, 4n_2-1), (n_3, n_2n_3-1), \dots, (n_k , n_{k-1}n_k-1), \dots , (n_r, n_{r-1}n_r - 1) .\]
    
    If $r = 2$ and $d = 2n_2$, then the cusp just has one Newton pair $(n_2, 4n_2 -1) = (d/2, 2d-1)$.  These are precisely the curves in Theorem \ref{thm:onepair}.2.  

    For $r = 3$, these are the curves given in Theorem \ref{thm:twopairs}.4.
\end{enumerate}

\begin{remark}
    By direct computation, the multiplicity sequence of the cusp of the first type of AMS curve is $\overline{m} = ((n_1-1)n_2\dots n_r, (n_2\dots n_r)_{2(n_1-1)}, \overline{m}')$ and that of the cusp of the second type is $\overline{m} = ((n_2\dots n_r)_3, \overline{m}')$.  In particular, Lemma \ref{lem:existenceofmostcurves} applies for all AMS curves with $k = n_1 - 1$, so they can be explicitly constructed from lower degree curves via an inductive process.  
\end{remark}

Note that in the construction of AMS curves, the order of the factorization matters: for $d$ even, $d \ge 6$, the factorization $d = 2n$ yields the cusp with a single Newton pair $(d/2, 2d-1)$, but the factorization $d = n \cdot 2$ yields a cusp with two Newton pairs $(n-1,n),(2,d-1)$. In other words, each ordered factorization of $d$ gives rise to a topologically distinct unicupsidal plane curve.  We state this observation as a result.

 \begin{theorem}(= Theorem \ref{main2})
    For any degree $d$, there are at least as many non-isomorphic rational unicuspidal plane curves as there are ordered factorizations of $d$. 
 \end{theorem}

 \begin{remark}
     Counting the number of ordered factorizations of an integer $n$ is known as Kalm\'ar's problem \cite{Kalmar, Hwang}.  The associated sequence $m(n)$ of ordered factorizations of $n$ can be found in \cite[A074206]{oeis} and has surprising connections to the Riemann zeta function.
 \end{remark}

This is evidence of a relationship between \textit{non}-primeness of degree with many different cuspidal curves of that degree.  See Section \ref{s:prime} for more on prime degree curves.

\subsubsection{Kashiwara curves}\label{kashiwaracurves}

The Kashiwara curves are unicuspidal rational curves with Kodaira dimension $\kappa(\bP^2 - C) = -\infty$ that are not AMS curves.  They come in six families.  Following the notation in \cite{FdBLVMHN1}, we enumerate them below.  In what follows, $N \ge 1$ is a positive integer, $l \ge 0$ is a nonnegative integer, and $\lambda_i \ge 0$ is a nonnegative integer.  Furthermore, if $l = 0$, then $\lambda_i \ge 1$.

\begin{enumerate}
    \item Type $II_{ge}$.  The curves of this type have degree $d = \phi_{2l+3}\phi_{2l+5}$ and cusp parameterized by a single Newton pair $(\phi_{2l+3}^2, \phi_{2l+5}^2)$.  These are the curves in Theorem \ref{thm:onepair}.3.
    \item Type $II_{sp}$.  The curves of this type has degree $d = \phi_{2l+3}$ and cusp parameterized by a single Newton pair $(\phi_{2l+1}, \phi_{2l+5})$.  These are the curves in Theorem \ref{thm:onepair}.4.
    \item Type $II^+_{ge}$. The curves of type $II^+(l,N; \lambda_1, \dots, \lambda_N)_{ge}$ have degree $d = \phi_{2l+3}\phi_{2l+5} n_1 \dots n_N$ where 
\[  n_i = \left\{ \begin{array}{c} 
\lambda_i \phi_{2l+3}^2 + \phi_{2l+3}\phi_{2l-1} -1 \text{ for } i \text{ odd } \\
    \lambda_i \phi_{2l+3}^2 + \phi_{2l+3}(\phi_{2l+3}-\phi_{2l-1}) -1 \text{ for } i \text{ even }
\end{array} \right. \]

The cusp is parameterized by $N+1$ Newton pairs: 
\[\left(n_1, \tfrac{\phi_{2l+5}^2n_1-1}{\phi_{2l+3}^2} \right) , \left(n_2 , \tfrac{n_1n_2-1}{\phi_{2l+3}^2} \right), \dots, \left(n_k , \tfrac{n_{k-1}n_k-1}{\phi_{2l+3}^2} \right) , \dots,  \left(  n_N, \tfrac{n_{N-1}n_N - 1}{\phi_{2l+3}^2} \right) ,\left(\phi_{2l+3}^2, n_N \right) .\]

    \item Type $II^+_{sp}$.  The curves of type $II^+(l,N; \lambda_1, \dots, \lambda_N)_{sp}$ have degree $d = \phi_{2l+5} n_1 \dots n_N$ where 
\[  n_i = \left\{ \begin{array}{c} 
\lambda_i \phi_{2l+3}^2 + \phi_{2l+3}\phi_{2l-1} -1 \text{ for } i \text{ odd } \\
    \lambda_i \phi_{2l+3}^2 + \phi_{2l+3}(\phi_{2l+3}-\phi_{2l-1}) -1 \text{ for } i \text{ even }
\end{array} \right. \]

The cusp is parameterized by $N+1$ Newton pairs:

\[ \left( n_1, \tfrac{\phi_{2l+5}^2n_1-1}{\phi_{2l+3}^2} \right) , \left( n_2, \tfrac{n_1n_2-1}{\phi_{2l+3}^2} \right), \dots,  \left( n_k, \tfrac{n_{k-1}n_2k-1}{\phi_{2l+3}^2} \right), \dots, \left(  n_N, \tfrac{n_{N-1}n_N - 1}{\phi_{2l+3}^2} \right),  \left( \phi_{2l+3}, \tfrac{n_N+1}{\phi_{2l+3}} \right) .  \]

    \item Type $II^-_{ge}$.  The curves of type $II^-(l,N; \lambda_1, \dots, \lambda_N)_{ge}$ have degree $d = \phi_{2l+3}\phi_{2l+1} n_1 \dots n_N$ where 
\[  n_i = \left\{ \begin{array}{c} 
\lambda_i \phi_{2l+3}^2 + \phi_{2l+3}\phi_{2l-1} -1 \text{ for } i \text{ even } \\
    \lambda_i \phi_{2l+3}^2 + \phi_{2l+3}(\phi_{2l+3}-\phi_{2l-1}) -1 \text{ for } i \text{ odd }
\end{array} \right. \]

The cusp is parameterized by $N+1$ Newton pairs: 
\[ \left( n_1, \tfrac{\phi_{2l+1}^2n_1-1}{\phi_{2l+3}^2} \right) ,  \left( n_2, \tfrac{n_1n_2-1}{\phi_{2l+3}^2} \right), \dots,  \left( n_k, \tfrac{n_{k-1}n_k-1}{\phi_{2l+3}^2} \right) , \dots,  \left(   n_N, \tfrac{n_{N-1}n_N - 1}{\phi_{2l+3}^2} \right) , \left( \phi_{2l+3}^2, n_N \right) .  \]

    \item Type $II^-_{sp}$. The curves of type $II^-(l,N; \lambda_1, \dots, \lambda_N)_{sp}$ have degree $d = \phi_{2l+1} n_1 \dots n_N$ where 
\[  n_i = \left\{ \begin{array}{c} 
\lambda_i \phi_{2l+3}^2 + \phi_{2l+3}\phi_{2l-1} -1 \text{ for } i \text{ even } \\
    \lambda_i \phi_{2l+3}^2 + \phi_{2l+3}(\phi_{2l+3}-\phi_{2l-1}) -1 \text{ for } i \text{ odd }
\end{array} \right. \]

The cusp is parameterized by $N+1$ Newton pairs: 
\[ \left(n_1, \tfrac{\phi_{2l+1}^2n_1-1}{\phi_{2l+3}^2}\right), \left(n_2, \tfrac{n_1n_2-1}{\phi_{2l+3}^2}\right), \dots,  \left(n_k, \tfrac{n_{k-1}n_k-1}{\phi_{2l+3}^2}\right),\dots, \left(  n_N, \tfrac{n_{N-1}n_N - 1}{\phi_{2l+3}^2}\right), \left(\phi_{2l+3}, \tfrac{n_N+1}{\phi_{2l+3}}\right) .  \]
    
\end{enumerate}

\begin{remark}
    In \cite{FdBLVMHN1}, the types are enumerated by a different set of invariants, the \textit{decorations of the Eisenbud-Neumann splice diagrams}, but these can be converted to the Newton pairs.  We illustrate this below for Type $II^+_{sp}$ and the process is similar for all other types.  From the invariants $p$ and $a$ given in \cite[\S 6.3]{FdBLVMHN1} to describe the Kashiwara curves, we compute the relevant Newton pairs $(p_1, q_1), ..., (p_{N+1}, q_{N+1})$ using the formulas $q_1 = a_1$, $q_k = a_k - p_kp_{k-1}a_{k-1}$ for $k > 1$ (see \cite[\S 2.1]{FdBLVMHN1}).

From \cite[\S 6.3]{FdBLVMHN1}, the Kashiwara curves of type $II^+_{sp}$ are given by $(p_i, a_i)$ where:
\begin{align*}
        p_k &= n_k \text{ for } 1 \leq k \leq N \text{, and } p_{N+1} = \phi_{2l+3}\\
        a_k &= \frac{\phi_{2l+5}^2n_1^2...n_{k-1}^2n_k-1}{\phi_{2l+3}^2} \text{ for } 1 \leq k \leq N \text{, and } a_{N+1} = \frac{\phi_{2l+5}n_1^2...n_N^2+1}{\phi_{2l+3}}
\end{align*}

Therefore, 
\begin{align*}
 \text{for } k = 1: \quad   q_1 &= a_1 \\
    &= \frac{\phi^2_{2l+5}n_1-1}{\phi_{2l+3}^2}\\
 \text{for }   2 \leq k \leq N: \quad q_k &= a_k - p_kp_{k-1}a_{k-1}\\
    &=\frac{\phi_{2l+5}^2n_1^2...n_{k-1}^2n_k-1}{\phi_{2l+3}^2}-n_kn_{k-1}\left(\frac{\phi_{2l+5}^2n_1^2...n_{k-2}^2n_{k-1}-1}{\phi_{2l+3}^2}\right)\\
    &=\frac{\phi_{2l+5}^2n_1^2...n_{k-1}^2n_k-1-\phi_{2l+5}^2n_1^2...n_{k-1}^2n_k-n_kn_{k-1}}{\phi_{2l+3}^2}\\
    &=\frac{n_kn_{k-1}-1}{\phi_{2l+3}^2}\\
  \text{for }  k = N+1: \quad q_{N+1} &= a_{N+1}-p_{N+1}p_Na_N\\
    &=\frac{\phi_{2l+5}^2n_1^2...n_N^2+1}{\phi_{2l+3}}-\phi_{2l+3}n_N\left(\frac{\phi_{2l+5}^2n_1^2...n_{N-1}^2n_N-1}{\phi^2_{2l+3}}\right)\\
    &=\frac{\phi_{2l+5}n_1^2...n_N^2+1-\phi_{2l+5}n_1^2...n_N^2+n_N}{\phi_{2l+3}}\\
    &=\frac{n_N+1}{\phi_{2l+3}}
\end{align*}
and we obtain the Newton pairs as claimed.
\end{remark}

\subsection{Kodaira dimension $0$}

As proved by Tsunoda in \cite{Tsunoda}, there exist no unicuspidal rational curves with $\kappa(\bP^2 - C) = 0$.

\subsection{Kodaira dimension $1$}\label{tonocurves}

In \cite{Tono2001}, Tono has classified all unicuspidal rational curves with $\kappa(\bP^2 - C) = 1$.  There are four families.  We follow the notation in \cite{Tono2001}.

\begin{enumerate}
    \item Type I(a).  These curves have degree $d = a^2 + 1$ where $a \ge 3$ and the parameterization  of the cusp is given by the two Newton pairs: 
\[ (a-1, a), (a, (a+1)^2) .\]

 These curves are precisely the curves in Theorem \ref{thm:twopairs}.5.
 
    \item Type I(b). These curves have degree  $d = a^2s + 1$ where $a \ge 3$ and $s \ge 2$ and the parameterization of the cusp is given by three Newton pairs:
\[ (a-1, a),(s, as+1), (a, as+1) .\]

 The only curves of degree at most 30 of Type I(b) are the degree 19 curve and the first degree 28 curve in Theorem \ref{thm:threepairs}.

 \item Type II(a). These curves have degree $d = 8n^2+4n+1$ where $n \ge 2$.  The parameterization of the cusp is given by the two Newton pairs: 
\[(n, 4n+1), (4n+1, (2n+1)^2) . \]
 These curves are precisely the curves in Theorem \ref{thm:twopairs}.6.

 \item Type II(b). These curves have degree $d = 2(4n+1)^2s-4n(2n+1)$ where $n \ge 2$ and $s \ge 1$.  The parameterization of the cusp is given by the three Newton pairs: 
\[(n(4s-1),(s-1)(4n+1)), (4s-1, (4n+1)s-n), (4n+1, (4n+1)s-n) . \]

 Because all curves of Type II(b) have degree greater than 30, they do not appear in the classification in Theorem \ref{thm:threepairs}.
\end{enumerate}

\begin{remark}
    Both Type I(a) and I(b) curves can be constructed by blow ups and blow downs from the plane curve $zx^{a-1} = y^a$ using Lemma \ref{lem:type1curves}.
\end{remark}

\subsection{Kodaira dimension $2$}\label{orevkovcurves}

There exists no complete classification of curves with $\kappa(\bP^2 - C) = 2$.  In fact, there are only two families of known examples, due to Orevkov in \cite{Orevkov}.  

\begin{enumerate}
    \item Family $C_{4k}$.  For $k = 1$, the curve $C_4$ is a degree $8$ curve with cusp parameterized by a single Newton pair $(3,22)$.  For $k > 1$, the curve $C_{4k}$ has degree $\phi_{4k+2}$ and has cusp given by the two Newton pairs $(\frac{\phi_{4k}}{3}, \frac{\phi_{4k+4}}{3}), (3,1)$.  These curves correspond to the curves in Theorem \ref{thm:onepair}.5 and Theorem \ref{thm:twopairs}.7, for $k = 1$ and $k > 1$ respectively. 

    \item Family $C_{4k}^*$.  For $k = 1$, the curve $C_4^*$ is a degree $16$ curve with cusp parameterized by a single Newton pair $(6,43)$.  For $k > 1$, the curve $C_{4k}^*$ has degree $2\phi_{4k+2}$ and has cusp given by the two Newton pairs $(\frac{\phi_{4k}}{3}, \frac{\phi_{4k+4}}{3}), (6,1)$. These curves correspond to the curves in Theorem \ref{thm:onepair}.6 and Theorem \ref{thm:twopairs}.8, for $k = 1$ and $k > 1$ respectively. 
\end{enumerate}

\subsubsection{A question on many Newton pairs}

For each family in the classification by Kodaira dimension parameterized by at most two Newton pairs or of degree at most 30, we have connected it with Section \ref{s:newtonpairs}.  Observe that all known curves with $\kappa(\bP^2 - C) \ge 0$ are parameterized by at most three Newton pairs.  This motivates the following question.

\begin{question}
      If $C$ is a rational unicuspidal plane curve parameterized by at least four Newton pairs, then is $\kappa(\bP^2  -C) = - \infty$?  In particular, is $C$ necessarily an AMS curve or a Kashiwara curve?   
\end{question}

Running a computer program suggests that the answer is yes. 

\subsection{Self-intersections and log canonical thresholds}

Motivated by the moduli perspective in the introduction, we conclude this section with several results on the self-intersection of $C$ in its minimal resolution, denoted by $\widetilde{C}^2 = n$, and its log canonical threshold $\lct(C)$.  These numbers have geometric significance and are especially important for moduli problems. For example, if $(\bP^2, \alpha C_d)$ is a pair in $\calM_d^{\alpha}$, we must have $\lct(C_d) \ge \alpha$ (if $\alpha \ge \frac{3}{d}$, this is by definition, and if $\alpha < \frac{3}{d}$, this is a consequence of Odaka's result that K-semistable pairs are klt \cite{Odaka}).  Furthermore, if $\widetilde{C}$ is to be contractible (as in Examples \ref{quintics} and \ref{octics}), then we must have $\widetilde{C}^2 < 0$.  

\begin{definition}
    Let $C \subset X$ be a curve in a smooth surface $X$.  The \textbf{log canonical threshold} of the curve, denoted $\lct(C)$, is 
    \[ \lct(C) = \sup_{r \in \bR} \{ (X,rC) \text{ is log canonical } \}. \]
\end{definition}

For the definition of log canonical singularity, see \cite{KollarMori}.  This is a measure of how `bad' a curve singularity is, with smaller log canonical threshold corresponding to a `worse' singularity.  If $C$ is smooth, then $\lct(C) = 1$.  For a unicuspidal curve with cusp $p \in C$ with Puiseux pairs $\{(P_i, Q_i) \}_{i = 1}^k$, the log canonical threshold is given by 

\begin{equation}\label{lctformula}
    \lct(C) = \frac{1}{P_1} + \frac{1}{Q_1}.
\end{equation}

\begin{definition}
    Let $C \subset X$ be a projective curve in a smooth projective surface.  Let $\pi: \widetilde{X} \to X$ be the minimal log resolution of the pair $(X,C)$, and let $\widetilde{C} = \pi^{-1}_*C$ be the strict transform of $C$.  The \textbf{self-intersection in the minimal resolution} is the integer $\widetilde{C}^2$.  We will refer to this number just as the \textbf{self-intersection}.   
\end{definition}

This self-intersection number is related to the embedding of $C$ in $X$ and the Kodaira dimension of the open surface $X - C$.  For a rational unicuspidal curve with Puiseux pairs $\{(P_i, Q_i)\}_{i=1}^k$, the self-intersection in the minimal resolution can be computed directly (see, e.g. \cite{Liu2014}): 

\begin{equation}\label{selfintersection}
    \widetilde{C}^2 = 3d - 1 - P_1 - \sum_{i=1}^k Q_i.
\end{equation}

For each of the curves in the previous section, we compute the log canonical threshold and the self intersection.

\subsubsection{Kodaira dimension $-\infty$}

\begin{prop}
    For all AMS curves of degree $d = n_1n_2 \dots n_r$, the log canonical threshold is $\frac{1}{(n_1-1)n_2\dots n_r} + \frac{1}{d}$ and the self-intersection is $n_r > 0$.  
\end{prop}

\begin{proof}
    We directly apply the formulas \ref{lctformula} and \ref{selfintersection}.  The Puiseux pairs $\{(P_i, Q_i)\}_{i=1}^k$ are obtained from the Newton pairs $\{(p_i,q_i)\}_{i=1}^k$ by the formulas $P_i = p_i p_{i+1} \dots p_k$ and $Q_i = q_i p_{i+1} \dots p_k$.  
\end{proof}

\begin{prop}
    For the Kashiwara curves, the log canonical thresholds and self-intersections are given in Table \ref{lctK}. 
\begin{center}
\begin{table}[h]
\caption{Log canonical thresholds and self-intersection of Kashiwara curves.}
\label{lctK}
    \begin{tabular}{|c|c|c|c|}
    \hline
     type & degree & log canonical threshold & self intersection  \\
     \hline Type $II_{ge}$ & $\phi_{2l+3}\phi_{2l+5}$ & $\frac{1}{\phi_{2l+3}^2} + \frac{1}{\phi_{2l+5}^2}$ & $0$  \\
    \hline Type $II_{sp}$ & $\phi_{2l+3}$ & $\frac{1}{\phi_{2l+1}} + \frac{1}{\phi_{2l+5}}$ & $-1$ \\
    \hline Type $II^+_{ge}$& $\phi_{2l+3}\phi_{2l+5}n_1\dots n_N$ & $\frac{1}{(n_1...n_N)\phi_{2l+3}^2}+\frac{1}{(\phi_{2l+5}^2n_1-1)(n_2...n_N)}$ & $0$  \\
    \hline Type $II^+_{sp}$ & $\phi_{2l+5}n_1\dots n_N$ & $\frac{1}{(n_1...n_N)\phi_{2l+3}}+\frac{\phi_{2l+3}}{(\phi_{2l+5}^2n_1-1)(n_2...n_N)}$ & $-1$ \\
    \hline Type $II^-_{ge}$ & $\phi_{2l+3}\phi_{2l+1}n_1\dots n_N$ & $\frac{1}{(n_1...n_N)\phi_{2l+3}^2}+\frac{1}{(\phi_{2l+1}^2n_1-1)(n_2...n_N)}$ & $0$  \\
    \hline Type $II^-_{sp}$& $\phi_{2l+1}n_1\dots n_N$ &  $\frac{1}{(n_1...n_N)\phi_{2l+3}}+\frac{\phi_{2l+3}}{(\phi_{2l+1}^2n_1-1)(n_2...n_N)}$ & $-1$ \\
    \hline
    \end{tabular}
\end{table}
\end{center}
\end{prop}

\begin{proof}
The log canonical threshold is a direct consequence of formula \ref{lctformula}.  We give the computation for the self-intersection of Kashiwara curves of type $II^+_{sp}$ as the computations for the other cases are similar.  This is an application of the formula \ref{selfintersection} and the Fibonacci identities 

\begin{equation}\label{fibonacci1}
    \phi_{n-2} + \phi_{n+2} = 3\phi_n
\end{equation}

\begin{equation}\label{fibonacci}
    \phi_{n}^2 - \phi_{n+r}\phi_{n-r} = (-1)^{n-r}\phi_r^2.
\end{equation}

By formula \ref{selfintersection}, $\widetilde{C}^2 = 3d-1-P_1-\sum_{i=1}^{N+1}Q_i$. For curves of type $II^+_{sp}$, we compute
\begin{align*}
    -Q' &= -Q_1 -Q_2 + ... -Q_{N+1}\\
    &= -\frac{\phi_{2l+5}^2(n_1...n_N)+(n_2...n_N)-(n_1...n_N)+(n_3...n_N)-\dots-(n_{N-1}n_N)+1-n_N-1}{\phi_{2l+3}} \\
    &= -\frac{(\phi_{2l+5}^2+1)(n_1...n_N)}{\phi_{2l+3}},\\
    \text{so } \widetilde{C}^2&=3d-1-P_1-Q'\\
    &=3\phi_{2l+5}(n_1...n_N)-1-\phi_{2l+3}(n_1...n_N) - \frac{(\phi_{2l+5}^2+1)(n_1...n_N)}{\phi_{2l+3}}\\
    &=\frac{3\phi_{2l+3}\phi_{2l+5}(n_1...n_N)-\phi_{2l+3}^2(n_1...n_N)-\phi_{2l+5}^2(n_1...n_N)-(n_1...n_N)}{\phi_{2l+3}}-1\\
    &=\frac{(n_1...n_N)}{\phi_{2l+3}}(3\phi_{2l+3}\phi_{2l+5} - \phi_{2l+3}^2 - \phi_{2l+5}^2-1) - 1\\
    &= \frac{(n_1...n_N)}{\phi_{2l+3}}(\phi_{2l+3}\phi_{2l+5} - (\phi_{2l+5}- \phi_{2l+3})^2-1) - 1\\
&= \frac{(n_1...n_N)}{\phi_{2l+3}}(\phi_{2l+3}\phi_{2l+5} - \phi_{2l+4}^2-1) - 1\\
&= \frac{(n_1...n_N)}{\phi_{2l+3}}(-(-1)^{2l+3}\phi_2-1) - 1 \text{ (by \ref{fibonacci})}\\
    &= \frac{(n_1...n_N)}{\phi_{2l+3}}(0)-1\\
    &=-1.
\end{align*}
\end{proof}

Following \cite{DS22}, rational unicuspidal plane curves with log canonical threshold at least $\frac{3}{d}$ are geometrically interesting as they give rise to smooth non-planar limits of plane curves contained non-normal surfaces (as in Example \ref{octics}).  The following question was posed in \cite[Question, p.39]{DS22}: 

\begin{question}\label{lctq}
    For what degree $d$ does there exist a unicuspidal rational plane curve of degree $d$ with log canonical threshold at least $\frac{3}{d}$?
\end{question}

We expect only few curves to have large log canonical threshold at least $\frac{3}{d}$.  No curves with $\kappa = -\infty$ have large $\lct$: 

\begin{lemma}
    For all curves with Kodaira dimension $-\infty$, the log canonical threshold is less than $\frac{3}{d}$.
\end{lemma}

\begin{proof}
The computation for AMS curves is straightforward.  The computation for Kashiwara curves again uses the Fibonacci identities \ref{fibonacci1}, \ref{fibonacci}, and we illustrate it below for curves of type $II^+_{sp}$ (the remaining cases are similar).  

By \ref{lctformula}, the log canonical threshold is $\frac{1}{P_1} + \frac{1}{Q_1}$.  We aim to show that this is less than $\frac{3}{d}$. For curves of type $II^+_{sp}$, the desired inequality is:
\begin{align*}
    &\frac{1}{P_1}+\frac{1}{Q_1} < \frac{3}{d}\\
    &\iff \frac{1}{(n_1...n_N)\phi_{2l+3}}+\frac{\phi_{2l+3}}{(\phi_{2l+5}^2n_1-1)(n_2...n_N)} < \frac{3}{\phi_{2l+5}(n_1...n_N)}\\
    &\iff \frac{1}{\phi_{2l+3}n_1}+\frac{\phi_{2l+3}}{\phi_{2l+5}^2n_1-1} < \frac{3}{\phi_{2l+5}n_1}\\
    &\iff \frac{\phi_{2l+5}^2n_1-1+\phi_{2l+3}^2n_1}{\phi_{2l+3}n_1(\phi_{2l+5}^2n_1-1)} < \frac{3}{\phi_{2l+5}n_1}\\
    &\iff \frac{n_1(\phi_{2l+5}^2+\phi_{2l+3}^2)-1}{\phi_{2l+3}n_1(\phi_{2l+5}^2n_1-1)} < \frac{3}{\phi_{2l+5}n_1}\\    
    &\iff \phi_{2l+5}(n_1(\phi_{2l+5}^2 + \phi_{2l+3}^2)-1) < 3\phi_{2l+3}(\phi_{2l+5}^2n_1-1)\\
    &\iff \phi_{2l+5}\phi_{2l+3}^2n_1 + \phi_{2l+5}(\phi_{2l+5}^2n_1 -1) < 3\phi_{2l+3}(\phi_{2l+5}^2n_1-1)\\
    &\iff \phi_{2l+5}\phi_{2l+3}^2n_1 < (3\phi_{2l+3} - \phi_{2l+5})(\phi_{2l+5}^2n_1-1)\\
    &\iff \phi_{2l+5}\phi_{2l+3}^2n_1 < \phi_{2l+1}(\phi_{2l+5}^2n_1-1) \text{ (by \ref{fibonacci1})}\\
    &\iff \phi_{2l+1} < \phi_{2l+1}\phi_{2l+5}^2n_1-\phi_{2l+5}\phi_{2l+3}^2n_1\\
    &\iff \phi_{2l+1} < \phi_{2l+5}n_1(\phi_{2l+1}\phi_{2l+5}-\phi_{2l+3})\\
    &\iff \phi_{2l+1} < \phi_{2l+5}n_1(-1)(-1)^{2l+1} \text{ (by \ref{fibonacci})}\\
    &\iff \phi_{2l+1} < \phi_{2l+5}n_1,
\end{align*}
and as $n_1 \ge 1$ and $l \ge 0$, this last inequality holds.  Therefore, the log canonical threshold is less than $\frac{3}{d}$. 
\end{proof}

\subsubsection{Kodaira dimension $1$}

\begin{prop}
    For the unicuspidal rational curves with Kodaira dimension $1$, the log canonical thresholds and self intersections are given in the Table \ref{lctT}. 
\begin{center}
\begin{table}[h]
\caption{Log canonical thresholds and self-intersection of Tono curves.}
\label{lctT}
    \begin{tabular}{|c|c|c|c|}
\hline type & degree & log canonical threshold & self intersection \\
    \hline Type I(a) & $a^2 + 1, a \ge 3$ & $\frac{1}{a(a-1)} + \frac{1}{a^2}$ &  $1 - a \le -2$\\
    \hline Type I(b) & $a^2s + 1, a \ge 3$ & $\frac{1}{as(a-1)} + \frac{1}{a^2s}$ &  $1-a \le -2$ \\
    \hline Type II(a) & $8n^2+4n+1, n \ge 2$ & $\frac{1}{n(4n+1)} + \frac{1}{(4n+1)^2}$ &  $-n \le -2$ \\
    \hline Type II(b) & $2(4n+1)^2s - 4n(2n+1), n \ge 2$ & $\frac{1}{n(4n+1)(4s-1)} + \frac{1}{(4n+1)^2(s-1)}$ & $-n \le -2 $ \\
    \hline
    \end{tabular}
\end{table}
\end{center} 
\end{prop}

\begin{proof}
    The proof is again a direct application of \ref{lctformula} and \ref{selfintersection}.
\end{proof}

By a straightforward calculation, we observe the following, related to Question \ref{lctq}.

\begin{lemma}
    All unicuspidal rational curves of Kodaira dimension $1$ have log canonical threshold less than $\frac{3}{d}$.
\end{lemma}

\subsubsection{Kodaira dimension $2$}

In each case for $\kappa = 2$, it is proven in \cite{Orevkov} that $\widetilde{C}^2 = - 2$.  In fact, \cite{Tono2012} proves that these are the only curves with $\kappa = 2$ and $\widetilde{C}^2 \ge -2$, i.e. if there exist other rational unicuspidal plane curves with $\kappa = 2$, they must have $\widetilde{C}^2 < -2$.  

Furthermore, these are the only known examples with $\lct > 3/d$.  We summarize the known properties (which are verifiable by straightforward computation): 

\begin{prop}
    For the unicuspidal rational curves with Kodaira dimension $2$, the log canonical thresholds and self intersections are given in Table \ref{lctO}. 
\begin{center}
\begin{table}[h]
\caption{Log canonical thresholds and self-intersection of Orevkov curves.}
\label{lctO}
    \begin{tabular}{|c|c|c|c|}
\hline type & degree & log canonical threshold & self intersection \\
    \hline $C_{4}$ & $8$ & $\frac{1}{3} + \frac{1}{22}$ &   $-2$\\
    \hline $C_{4k}$, $k > 1$ & $\phi_{4k+2}$ & $\frac{1}{\phi_{4k}} + \frac{1}{\phi_{4k+4}}$ &  $ -2$ \\
    \hline $C_{4}^*$ & $16$ & $\frac{1}{6} + \frac{1}{43}$ &  $-2$ \\
    \hline $C_{4k}^*$, $k > 1$ & $2\phi_{4k+2}$ & $\frac{1}{2\phi_{4k}} + \frac{1}{2\phi_{4k+4}}$ & $ -2 $ \\
    \hline
    \end{tabular}
\end{table}
\end{center} 
For every curve in this list, the log canonical threshold is \textit{greater} than $\frac{3}{d}$ (c.f Question \ref{lctq}). 
\end{prop}

\section{Remarks on prime degree curves}\label{s:prime}

Motivated by \cite{DS22}, we collect remarks on prime degree rational unicuspidal plane curves below.  

\begin{prop}\label{primedegree}
    Prime degree unicuspidal rational curves with $\kappa(\bP^2 - C) < 2$ or with cusp parameterized by at most two Newton pairs must be in the following list:
        \begin{enumerate}
            \item $d = p$ is any prime number, and the cusp of $C$ is parameterized by a single Newton pair $(d-1,d)$ (as in Theorem \ref{thm:onepair}.1),
            \item $d = \phi_j$ where $j \ge 5$ is an odd prime, and the cusp of $C$ is parameterized by a single Newton pair $(\phi_{j-2},\phi_{j+2})$ (as in Theorem \ref{thm:onepair}.4),
            \item $d = a^2s + 1$, where $d$ is prime and $a \ge 3$, and the cusp of $C$ is parameterized by two Newton pairs $(a-1,a),(a,(a+1)^2)$ (as in Theorem \ref{thm:twopairs}.5) if $s = 1$, or by three Newton pairs $(a-1,a), (s, as+1), (a,as+1)$ if $s > 1$,
            \item $d = 8n^2 + 4n + 1$, where $d$ is prime and $n \ge 2$, and the cusp of $C$ is parameterized by two Newton pairs $(n,4n+1),(4n+1,(2n+1)^2)$ (as in Theorem \ref{thm:twopairs}.6).
        \end{enumerate}
\end{prop}

\begin{proof}
    Every curve with $\kappa(\bP^2 - C)$ is either an AMS curve, a Kashiwara curve, or one of Tono's curves with $\kappa = 1$.  The only AMS curves of prime degree are those in Theorem \ref{thm:onepair}.1, and the only Kashiwara curves of prime degree are type $II_{sp}$, those in Theorem \ref{thm:onepair}.4.  The restriction that $j$ must be prime in this case follows because the Fibonacci number $\phi_j$ is divisible by $\phi_k$ for any $k$ dividing $j$.  Tono's curves with $\kappa = 1$ of type I(a), I(b), and II(a) give the remaining items in the list, noting that the curves of type II(b) have even degree.

    To complete the proof, we must then justify that any curve with $\kappa = 2$ parameterized by at most two Newton pairs does not have prime degree.  The only such curves are the degree $8$ or $16$ curve in Theorem \ref{thm:onepair}.5 or .6, which clearly have composite degree, or the degree $d = \phi_{4k+2}$ or $d = 2\phi_{4k+2}$ curves, where $k \ge 2$, in Theorem \ref{thm:twopairs}.7 or .8, which have composite degree because $\phi_{4k+2}$ is divisible by $\phi_{2k+1} >1$.
\end{proof}

Speculating that the Orevkov curves are the only rational unicuspidal curves with $\kappa = 2$, this is a complete list of possibilities for prime degree rational unicuspidal curves.  

In fact, the possible `nontrivial' prime degree curves (not isomorphic to $x^{p-1}z = y^d$) conjecturally come in infinite families. 

The first type of nontrivial prime degree curve are those of degree $\phi_j$, where $\phi_j$ is a Fibonacci number.  A famous conjecture in number theory is: 

\begin{conj}\label{inffibonacci}(Infinitude of Fibonacci primes.)
    There are infinitely many Fibonacci numbers that are prime.
\end{conj}

If true, this asserts that there are infinitely many prime degrees for which there exists a unicuspidal rational curve with a cusp as in Proposition \ref{primedegree}.2. 

Similarly, another famous conjecture in number theory, posed by Landau in 1912 is: 

\begin{conj}\label{infn2}(Landau's Fourth Problem.)
    There are infinitely many primes of the form $n^2 + 1$. 
\end{conj}

If true, this asserts the infinitude of prime degrees for which there exists a curve with cusp as in Proposition \ref{primedegree}.3.  While this conjecture is unsolved, by \cite{BaierZhao}, there are infinitely many primes of the form $a^2 s + 1$ (when one allows both $a$ and $s$ to vary).  Therefore, there do exist infinitely many prime degrees for which there is a nontrivial rational unicuspidal plane curve of Type I(a) or Type I(b) in Tono's classification.  This is the content of the following statement.

 \begin{theorem}(= Theorem \ref{main3})
     There are infinitely many primes $p$ such that there exists a unicuspidal rational curve $C$ of degree $p$ with Kodaira dimension $\kappa (\bP^2 - C) = 1$.
 \end{theorem}

 \begin{example}
    Two such curves are the degree $19 = 3^2 \cdot 2 + 1$ curve and first degree $28 = 3^2 \cdot 3 + 1$ curve in Theorem \ref{thm:threepairs}.  The theorem asserts that there are infinitely many others. 
\end{example}

More generally than the case $n^2 + 1$ or $sn^2 + 1$, given any polynomial $f(n)$ with integer coefficients, one may ask how often $f(n)$ is prime for integer input $n$.  One conjecture in this direction is the Bunyakovsky conjecture: 

\begin{conj}(Bunyakovsky Conjecture.)
Suppose $f(n)$ is a polynomial with integer coefficients such that: 
    \begin{enumerate}
        \item the leading coefficient is positive; 
        \item the polynomial is irreducible over $\bQ$; and 
        \item the values $\{ f(1), f(2), f(3), \dots \}$ have no common factor.
    \end{enumerate}
Then, $f(n)$ is prime infinitely often.
\end{conj}

There are in fact further conjectures on the asymptotic number of integers $n$ such that $f(n)$ is prime (part of the broader Bateman-Horn conjecture).  For our purposes, this conjecture would imply the infinitude of prime degrees giving Tono curves of Type I(b) and Type II(a). 

\begin{lemma}
    If $f(n) = sn^2 + 1$, the conditions of Bunyakovsky's conjecture are satisfied. 
\end{lemma}

\begin{proof}
    (1) and (2) are clear so we prove (3).  If $s = 1$, then $\gcd(f(1),f(2)) = 1$, so (3) holds.  If $s = 2$, then $\gcd(f(1), f(3)) = 1$ so again (3) holds.  Now, assume $s \ge 3$.  If $s \ne 2 \pmod 3$, then \[\gcd(f(1),f(2)) = \gcd(s+1, 4s+1) = \gcd(s+1,s-2) = \gcd(s-2,3) = 1\] so (3) holds.  If $s = 2 \pmod 3$, then $\gcd(f(1),f(2)) = 3$, but \[\gcd(f(2),f(3)) = \gcd(4s+1,9s+1) = \gcd(4s+1,s-1) = \gcd(s-1,5)\] so $\gcd(f(2),f(3))$ must equal $1$ or $5$.  This implies that $f(1), f(2),$ and $f(3)$ have no common factors, as desired. 
\end{proof}

Therefore, if Bunyakovsky's conjecture holds, this implies that there are infinitely many primes of the form $a^2 s + 1$ for any \textit{fixed} $s$, corresponding to rational unicuspidal curves of degree $d = a^2 s + 1$ of Type I(a) and Type I(b) in Tono's classification.

\begin{lemma}
    If $f(n) = 8n^2+4n + 1$, the conditions of Bunyakovsky's conjecture are satisfied. 
\end{lemma}

\begin{proof}
    (1) and (2) are clear, and $f(1) = 13$ and $f(2) = 41$, so $f(1)$ and $f(2)$ have no common factors, thus (3) holds. 
\end{proof}

Therefore, if Bunyakovsky's conjecture holds, there are also infinitely many primes of the form $8n^2 + 4n + 1$ corresponding to rational unicuspidal curves of degree $d = 8n^2 + 4n +1$ of Type II(a) in Tono's classification.

\section{Table of low degree rational unicuspidal plane curves}\label{s:table}

The results of the first three sections are combined in the following Theorem. 

\begin{theorem}(= Theorem \ref{main1})
    For $ d \le 30$, the following table gives an exhaustive list of all rational unicuspidal degree $d$ plane curves and their properties. See \cref{Footnote} for a description of what is meant by the parameterization column. 

\renewcommand*{\arraystretch}{1.4}
\begin{center}
{\scriptsize
\begin{longtable}{|c|c|c|c|c|c|c|c|}
\caption{All rational unicuspidal plane curves of degree $\le 30$.}\label{t:bigtable}\\

    \hline degree & Newton pair(s) & parametrization &  $\overline{m}$ & $\kappa$ & $\lct(C)$ & $\widetilde{C}^2$ & reference \\
    \hline $d \ge 3$ & $(d-1,d)$ & $  (t^{d-1},t^{d})$ & $(d)$ &  $- \infty$ & $\frac{1}{d-1} + \frac{1}{d}$ & $d$ & \ref{thm:onepair}.1, \S \ref{amscurves} \\
    \hline $\begin{array}{c}
         d = 2n  \\
         n \ge 2 
    \end{array}$ & $(d/2,2d-1)$ & $  (t^{d/2},t^{2d-1} )$  & $(d/2_3,d/2-1)$ & $- \infty$ & $\frac{2}{d}+\frac{1}{2d-1}$ & $d/2$ & \ref{thm:onepair}.2, \S \ref{amscurves}  \\

    \hline 5 & $(2,13)$ & $  (t^2, t^{13})$& $(2_6)$ & $- \infty$ & $\frac{1}{2}+\frac{1}{13}$ & -1 & \ref{thm:onepair}.4, \S \ref{kashiwaracurves} \\
    \hline 6 & $(2,3),(2,5)$ & $  (t^4, t^6+t^{11})$ & $(4,2_4)$ &$- \infty$ & $\frac{1}{4}+\frac{1}{6}$ & $2$ & \ref{thm:twopairs}.3, \S \ref{amscurves} \\
    \hline 8 & $(3,22)$ & $  (t^3, t^{22})$ & $(3_7)$ & $2 $ & $\frac{1}{3}+\frac{1}{22}$ & $-2$ & \ref{thm:onepair}.5, \S \ref{orevkovcurves} \\
    \hline 8 & $(2,7),(2,3)$ & $  (t^4, t^{14}+t^{17})$ & $(4_3, 2_3)$& $- \infty$ & $\frac{1}{4}+\frac{1}{14}$ & $2$ & \ref{thm:twopairs}.4, \S \ref{amscurves}  \\
    \hline 8 &  $(3,4),(2,7)$  & $  (t^6, t^8+t^{15})$& $(6, 2_6)$& $- \infty$ & $\frac{1}{6}+\frac{1}{8}$ & $2$ & \ref{thm:twopairs}.3, \S \ref{amscurves}  \\
    \hline 9 & $(2,3),(3,8)$  & $  (t^6, t^9+t^{17})$ & $(6, 3_4, 2)$& $- \infty$ & $\frac{1}{6}+\frac{1}{9}$ & $3$ & \ref{thm:twopairs}.3, \S \ref{amscurves}  \\
    \hline 10 & $(4,25)$  & $  (t^4, t^{25})$ & $(4_6)$ & $- \infty$ & $\frac{1}{4}+\frac{1}{25}$ & $0$ & \ref{thm:onepair}.3, \S \ref{kashiwaracurves} \\
    \hline 10 & $(2,3), (3,16)$  & $  (t^6, t^9+t^{25})$& $(6, 3_7)$& $1$ & $\frac{1}{6}+\frac{1}{9}$ & $-2$ & \ref{thm:twopairs}.5, \S \ref{tonocurves} \\
    \hline 10 & $(4,5), (2,9)$  & $  (t^8, t^{10}+t^{19})$&$(8, 2_8)$& $- \infty$ & $\frac{1}{8}+\frac{1}{10}$ & $2$ & \ref{thm:twopairs}.3, \S \ref{amscurves}  \\
    \hline 12 & $(2,7),(3,5)$ & $  (t^6, t^{21}+t^{26})$ & $(6_3, 3_3, 2)$&$- \infty$ & $\frac{1}{6}+\frac{1}{21}$ & $3$ & \ref{thm:twopairs}.4, \S \ref{amscurves}  \\
    \hline 12 & $(2,3),(4,11)$  & $  (t^8, t^{12}+t^{23})$&$(8, 4_4, 3)$ & $- \infty$ & $\frac{1}{8}+\frac{1}{12}$ & $4$ & \ref{thm:twopairs}.3, \S \ref{amscurves}  \\
    \hline 12 & $(3,11),(2,7)$  &$  (t^6, t^{22}+t^{29})$ & $(6_3, 4, 2_5)$ & $- \infty$& $\frac{1}{6}+\frac{1}{22}$ & $2$ & \ref{thm:twopairs}.4, \S \ref{amscurves}  \\
    \hline 12 & $(3,4),(3,11)$  &$  (t^9, t^{12}+t^{23})$ & $(9, 3_6, 2)$& $- \infty$  & $\frac{1}{9}+\frac{1}{12}$ & $3$ & \ref{thm:twopairs}.3, \S \ref{amscurves}  \\
    \hline 12 & $(5,6),(2,11)$  & $  (t^{10}, t^{12}+t^{23})$ & $(10, 2_{10})$& $- \infty$& $\frac{1}{10}+\frac{1}{12}$ & $2$ & \ref{thm:twopairs}.3, \S \ref{amscurves}  \\
    \hline 12 & $(2, 3), (2, 5), (2, 3)$  & $  (t^8, t^{12}+t^{22}+t^{25})$ & $(8, 4_4, 2_3)$ & $- \infty$ & $\frac{1}{8}+\frac{1}{12}$& $2$ & \ref{thm:threepairs}, \S \ref{amscurves} \\
    \hline 13 & $(5,34) $  & $  (t^5, t^{34})$ & $(5_6, 4) $& $- \infty$ & $\frac{1}{5}+\frac{1}{34}$ & $-1$ & \ref{thm:onepair}.4, \S \ref{kashiwaracurves} \\
    \hline 14 & $(6,7),(2,13)$  & $  (t^{12}, t^{14}+t^{27})$ &$(12, 2_{12})$& $- \infty$& $\frac{1}{12}+\frac{1}{14}$ & $2$ & \ref{thm:twopairs}.3, \S \ref{amscurves}  \\
    \hline 15 & $(2,3),(5,14)$ & $  (t^{10}, t^{15}+t^{29})$ &$(10, 5_4, 4)$& $- \infty$& $\frac{1}{10}+\frac{1}{15}$& $5$ & \ref{thm:twopairs}.3, \S \ref{amscurves} \\
    \hline 15 & $(4,5),(3,14)$  & $  (t^{12}, t^{15}+t^{29})$& $(12, 3_8, 2)$& $- \infty$& $\frac{1}{12}+\frac{1}{15}$ & $3$ & \ref{thm:twopairs}.3, \S \ref{amscurves}  \\
    \hline 16 & $(6,43) $  & $  (t^6, t^{43})$ &$(6_7)$& $2$ & $\frac{1}{6}+\frac{1}{43}$ & $-2$& \ref{thm:onepair}.6, \S \ref{orevkovcurves} \\
    \hline 16 & $(2,7),(4,7)$  & $  (t^8, t^{28}+t^{35})$ & $(8_3, 4_3, 3)$& $- \infty$ & $\frac{1}{8}+\frac{1}{28}$ & $4$ & \ref{thm:twopairs}.4, \S \ref{amscurves} \\
    \hline 16 & $(3,4),(4,15)$  & $  (t^{12}, t^{16} + t^{31})$&$(12, 4_6, 3)$& $- \infty$ & $\frac{1}{12}+\frac{1}{16}$ & $4$ & \ref{thm:twopairs}.3, \S \ref{amscurves}  \\
    \hline 16 & $(4,15),(2,7)$ & $  (t^8, t^{30}+t^{37})$&$(8_3, 6, 2_6)$& $- \infty$& $\frac{1}{8}+\frac{1}{30}$ & $2$ & \ref{thm:twopairs}.4, \S \ref{amscurves}  \\
    \hline 16 & $(7,8),(2,15)$ & $  (t^{14}, t^{16}+t^{31})$ &$(14, 2_14)$& $- \infty$& $\frac{1}{14}+\frac{1}{16}$ & $2$ & \ref{thm:twopairs}.3, \S \ref{amscurves}  \\
    \hline 16 & $(2, 7), (2, 3), (2, 3)$  & $  (t^8, t^{28}+t^{34}+t^{37})$ & $(8_3, 4_3, 2_3)$ & $- \infty$ & $\frac{1}{8}+\frac{1}{28}$ & $2$ & \ref{thm:threepairs}, \S \ref{amscurves}  \\
    \hline 16 & $(3, 4), (2, 7), (2, 3)$  & $  (t^{12}, t^{16}+t^{30}+t^{33})$& $(12, 4_6, 2_3)$ & $- \infty$ & $\frac{1}{12}+\frac{1}{16}$ & $2$ & \ref{thm:threepairs}, \S \ref{amscurves}  \\
    \hline 17 & $(3,4), (4,25) $  & $  (t^{12}, t^{16} + t^{41})$ &$(12, 4_9)$& $1$ & $\frac{1}{12}+\frac{1}{16}$ & $-3$ & \ref{thm:twopairs}.5, \S \ref{tonocurves} \\
    \hline 18 & $(2,3),(6,17)$  & $  (t^{12}, t^{18}+t^{35})$ &$(12, 6_4, 5)$& $- \infty$ & $\frac{1}{12}+\frac{1}{18}$ &  $6$ & \ref{thm:twopairs}.3, \S \ref{amscurves}  \\
    \hline 18 & $(3,11),(3,8)$  & $  (t^9, t^{33} + t^{41})$ &$(9_3, 6, 3_4, 2)$& $- \infty$& $\frac{1}{9}+\frac{1}{33}$ & $3$ & \ref{thm:twopairs}.4, \S \ref{amscurves}  \\
    \hline 18 & $(5,6),(3,17)$ & $  (t^{15}, t^{18}+t^{35})$ &$(15, 3_{10}, 2)$ & $- \infty$ & $\frac{1}{15}+\frac{1}{18}$ & $3$ & \ref{thm:twopairs}.3, \S \ref{amscurves} \\
    \hline 18 & $(8,9),(2,17)$ & $  (t^{16}, t^{18}+t^{35})$ & $(16, 2_{16})$& $- \infty$& $\frac{1}{16}+\frac{1}{18}$ & $2$ & \ref{thm:twopairs}.3, \S \ref{amscurves}  \\
    \hline 18 & $(2, 3), (2, 5), (3, 5)$  &$  (t^{12}, t^{18}+t^{33}+t^{38})$ & $(12, 6_4, 3_3, 2)$ & $- \infty$ & $\frac{1}{12}+\frac{1}{18}$ &  $3$ & \ref{thm:threepairs}, \S \ref{amscurves}  \\
    \hline 18 & $(2, 3), (3, 8), (2, 5)$  & $  (t^{12}, t^{18}+t^{34}+t^{39})$ & $(12, 6_4, 4, 2_4)$ & $- \infty$ & $\frac{1}{12}+\frac{1}{18}$ & $2$ & \ref{thm:threepairs}, \S \ref{amscurves}  \\
    \hline 19 & $(2, 3), (2, 7), (3, 7)$ & $  (t^{12}, t^{18} + t^{39} + t^{46})$ &  $(12, 6_5, 3_4)$ & $1$& $\frac{1}{12}+\frac{1}{18}$ & $-2$ & \ref{thm:threepairs}, \S \ref{tonocurves} \\
    \hline 20 &  $(2,7),(5,9)$  & $  (t^{10}, t^{35}+t^{44})$ &$(10_3, 5_3, 4)$ & $- \infty$ & $\frac{1}{10}+\frac{1}{35}$ & $5$ & \ref{thm:twopairs}.4, \S \ref{amscurves}  \\
    \hline 20 &  $(3,4),(5,19)$  & $  (t^{15}, t^{20}+t^{39})$&$(15, 5_6, 4)$ & $- \infty$ & $\frac{1}{15}+\frac{1}{20}$ & $5$ & \ref{thm:twopairs}.3, \S \ref{amscurves}  \\
    \hline 20 &  $(4,5),(4,19)$ & $  (t^{16}, t^{20}+t^{39})$&$(16, 4_8, 3)$ & $- \infty$ & $\frac{1}{16}+\frac{1}{20}$ & $4$ & \ref{thm:twopairs}.3, \S \ref{amscurves}  \\
    \hline 20 &  $(5,19),(2,9)$  & $  (t^{10}, t^{38} + t^{47})$&$(10_3, 8, 2_8)$ & $- \infty$ & $\frac{1}{10}+\frac{1}{38}$ & $2$ & \ref{thm:twopairs}.4 \S \ref{amscurves}  \\
    \hline 20 &  $(9,10),(2,19)$ & $  (t^{18}, t^{20}+t^{39})$ &$(18, 2_{18})$ & $- \infty$ & $\frac{1}{18}+\frac{1}{20}$ & $2$ & \ref{thm:twopairs}.3, \S \ref{amscurves}  \\
    \hline 20 & $(4, 5), (2, 9), (2, 3)$ & $  (t^{16}, t^{20} + t^{38} + t^{41})$  & $(16, 4_8, 2_3)$ & $-\infty$& $\frac{1}{16}+\frac{1}{20}$ & $2$ & \ref{thm:threepairs}, \S \ref{amscurves}  \\

    \hline 21 & $(2,3),(7,20)$   &$  (t^{14}, t^{21}+t^{41})$ &$(14, 7_4, 6)$ & $- \infty$& $\frac{1}{14}+\frac{1}{21}$  & $7$ & \ref{thm:twopairs}.3, \S \ref{amscurves} \\
    \hline 21 & $(6,7),(3,20)$   & $  (t^{18}, t^{21}+t^{41})$ &$(18, 3_{12}, 2)$ & $- \infty$ & $\frac{1}{18}+\frac{1}{21}$ & $3$ & \ref{thm:twopairs}.3, \S \ref{amscurves}  \\
    \hline 22 & $(10,11),(2,21)$   & $  (t^{20}, t^{22}+t^{43})$&$(20, 2_{20})$ & $- \infty$ & $\frac{1}{20}+\frac{1}{22}$  & $2$ & \ref{thm:twopairs}.3, \S \ref{amscurves} \\
    \hline 24 &  $(2,3),(8,23)$  & $  (t^{16}, t^{24}+t^{47})$&$(16, 8_4, 7)$ & $- \infty$ & $\frac{1}{16}+\frac{1}{24}$ & $8$ & \ref{thm:twopairs}.3, \S \ref{amscurves}  \\
    \hline 24 &  $(2,7),(6,11)$  & $  (t^{12}, t^{42}+t^{53})$&$(12_3, 6_3, 5)$ & $- \infty$& $\frac{1}{12}+\frac{1}{42}$ 
 & $6$ & \ref{thm:twopairs}.4, \S \ref{amscurves} \\
    \hline 24 & $(3,4),(6,23)$   & $  (t^{18}, t^{24} + t^{47})$&$(18, 6_6, 5)$ & $- \infty$& $\frac{1}{18}+\frac{1}{24}$ 
 & $6$ & \ref{thm:twopairs}.3, \S \ref{amscurves}  \\
    \hline 24 &  $(3,11),(4,11)$  & $  (t^{12}, t^{44}+t^{55})$&$(12_3, 8, 4_4, 3)$ & $- \infty$& $\frac{1}{12}+\frac{1}{44}$ & $4$ & \ref{thm:twopairs}.4, \S \ref{amscurves}  \\
    \hline 24 & $(4,15),(3,11)$  & $  (t^{12}, t^{45}+t^{56})$ &$(12_3, 9, 3_6, 2)$ & $- \infty$& $\frac{1}{12}+\frac{1}{45}$  & $3$ & \ref{thm:twopairs}.4, \S \ref{amscurves}  \\
    \hline 24 & $(5,6),(4,23)$ & $  (t^{20}, t^{24}+t^{47})$&$(20, 4_{10}, 3)$ & $- \infty$& $\frac{1}{20}+\frac{1}{24}$ 
 & $4$ & \ref{thm:twopairs}.3, \S \ref{amscurves}  \\
    \hline 24 &  $(6,23),(2,11)$  & $  (t^{12}, t^{46}+t^{57})$&$(12_3, 10, 2_{10})$ & $- \infty$& $\frac{1}{12}+\frac{1}{46}$  &  $2$& \ref{thm:twopairs}.4, \S \ref{amscurves} \\
    \hline 24 &  $(7,8),(3,23)$  &$  (t^{21}, t^{24}+t^{47})$ &$(21, 3_{14}, 2)$ & $- \infty$ & $\frac{1}{21}+\frac{1}{24}$  &  $3$ & \ref{thm:twopairs}.3, \S \ref{amscurves} \\
    \hline 24 &  $(11,12),(2,23)$  & $  (t^{22}, t^{24} + t^{47})$&$(22, 2_{22})$ & $- \infty$& $\frac{1}{22}+\frac{1}{24}$ & $2$ & \ref{thm:twopairs}.3, \S \ref{amscurves}  \\
    \hline 24 & $(2, 7), (2, 3), (3, 5)$ & $  (t^{12}, t^{42} + t^{51} + t^{56})$  & $(12_3, 6_3, 3_3, 2)$ & $- \infty$& $\frac{1}{12}+\frac{1}{42}$  & $3$ & \ref{thm:threepairs}, \S \ref{amscurves} \\
    \hline 24 & $(2, 7), (3, 5), (2, 5)$ &  $  (t^{12}, t^{42} + t^{52} + t^{57})$  & $(12_3, 6_3, 4, 2_4)$ & $- \infty$& $\frac{1}{12}+\frac{1}{42}$  & $2$ &  \ref{thm:threepairs}, \S \ref{amscurves} \\
    \hline 24 & $(3, 11), (2, 5), (2, 3)$ &  $  (t^{12}, t^{44} + t^{54} + t^{57})$  & $(12_3, 8, 4_4, 2_3)$ & $- \infty$& $\frac{1}{12}+\frac{1}{44}$  & $2$ & \ref{thm:threepairs}, \S \ref{amscurves} \\
    \hline 24 & $(2, 3), (2, 5), (4, 7)$ & $  (t^{16}, t^{24} + t^{44} + t^{51})$ & $(16, 8_4, 4_3, 3)$ & $- \infty$ & $\frac{1}{16}+\frac{1}{24}$  & $4$ & \ref{thm:threepairs}, \S \ref{amscurves}\\
    \hline 24 & $(2, 3), (4, 11), (2, 7)$ & $  (t^{16}, t^{24} + t^{46} + t^{53})$  & $(16, 8_4, 6, 2_6)$ & $- \infty$& $\frac{1}{16}+\frac{1}{24}$  & $2$ & \ref{thm:threepairs}, \S \ref{amscurves} \\
    \hline 24 & $(3, 4), (2, 7), (3, 5)$ & $  (t^{18}, t^{24} + t^{45} + t^{50})$  & $(18, 6_6, 3_3, 2)$ & $- \infty$& $\frac{1}{18}+\frac{1}{24}$  & $5$ & \ref{thm:threepairs}, \S \ref{amscurves} \\
    \hline 24 & $(3, 4), (3, 11), (2, 5)$ & $  (t^{18}, t^{24} + t^{46} + t^{51})$  & $(18, 6_6, 4, 2_4)$ & $- \infty$ & $\frac{1}{18}+\frac{1}{24}$  & $2$ & \ref{thm:threepairs}, \S \ref{amscurves} \\
    \hline 24 & $(5, 6), (2, 11), (2, 3)$ & $  (t^{20}, t^{24} + t^{46} + t^{49})$  & $(20, 4_{10}, 2_3)$ & $- \infty$ & $\frac{1}{20}+\frac{1}{24}$  & $2$ & \ref{thm:threepairs}, \S \ref{amscurves} \\
    \hline 24 & $\begin{array}{c}
          (2,3),(2,5), \\
         (2,3),(2,3)
    \end{array}$ & $\begin{array}{c}
          (t^{16}, \qquad \qquad \\
         t^{24}+t^{44}+t^{50}+t^{53})
    \end{array}$ & $(16, 8_4, 4_3, 2_2, 3)$& $- \infty$ & $\frac{1}{16}+\frac{1}{24}$  & $2$ & \ref{thm:fourpairs}, \S \ref{amscurves} \\
    \hline 25 & $(4,5),(5,24)$ & $  (t^{20}, t^{25}+t^{49})$ &$(20, 5_8, 4)$ & $- \infty$& $\frac{1}{20}+\frac{1}{25}$ 
 & $5$ & \ref{thm:twopairs}.3, \S \ref{amscurves}  \\
    \hline 26 & $(4,5),(5,36) $  &$  (t^{20}, t^{25}+t^{61})$ &$(20, 5_{11})$ & $1$ & $\frac{1}{20}+\frac{1}{25}$  & $-4$ & \ref{thm:twopairs}.5, \S \ref{tonocurves} \\
    \hline 26 &  $(12,13),(2,25)$  & $  (t^{24}, t^{26}+t^{41})$&$(24, 2_{19})$ & $- \infty$ & $\frac{1}{24}+\frac{1}{26}$  & $2$ & \ref{thm:twopairs}.3, \S \ref{amscurves}  \\
    \hline 27 & $(2,3),(9,26)$   & $  (t^{18}, t^{27}+t^{53})$ &$(18, 9_4, 8)$ & $- \infty$ & $\frac{1}{18}+\frac{1}{27}$  & $9$ & \ref{thm:twopairs}.3, \S \ref{amscurves} \\
    \hline 27 &  $(8,9),(3,26)$  & $  (t^{24}, t^{27}+t^{53})$&$(24, 3_{16}, 2)$ & $- \infty$ & $\frac{1}{24}+\frac{1}{27}$  & $3$ & \ref{thm:twopairs}.3, \S \ref{amscurves} \\
    \hline 27 & $(2, 3), (3, 8), (3, 8)$ & $  (t^{18}, t^{27} + t^{51} + t^{59})$  & $(18, 9_4, 6, 3_6)$ & $- \infty$ & $\frac{1}{18}+\frac{1}{27}$  & $3$ & \ref{thm:threepairs}, \S \ref{amscurves} \\
    \hline 28 & $(2,7), (7,13) $  & $  (t^{14}, t^{49}+t^{62})$ &$(14_3, 7_3, 6)$ & $- \infty$ & $\frac{1}{14}+\frac{1}{49}$  & $7$ & \ref{thm:twopairs}.4, \S \ref{amscurves} \\
    \hline 28 & $(3,4), (7,27) $  & $  (t^{21}, t^{28}+t^{55})$&$(21, 7_6, 6)$ & $- \infty$ & $\frac{1}{21}+\frac{1}{28}$  & $3$ & \ref{thm:twopairs}.3, \S \ref{amscurves} \\
    \hline 28 & $(6,7), (4,27) $  & $  (t^{24}, t^{28}+t^{55})$&$(24, 4_{12}, 3)$& $- \infty$& $\frac{1}{24}+\frac{1}{28}$  & $4$ & \ref{thm:twopairs}.3, \S \ref{amscurves} \\
    \hline 28 & $(7,27), (2,11) $  & $  (t^{14}, t^{54}+t^{65})$&$(14_3, 12, 2_{11})$ & $- \infty$ & $\frac{1}{14}+\frac{1}{54}$  & $2$ &\ref{thm:twopairs}.4, \S \ref{amscurves} \\
    \hline 28 & $(13,14), (2,27) $  & $  (t^{26}, t^{28}+t^{55})$ &$(26, 2_{26})$ & $- \infty$& $\frac{1}{26}+\frac{1}{28}$  & $2$ & \ref{thm:twopairs}.3, \S \ref{amscurves} \\
    \hline 28 & $(2, 3), (3, 10), (3, 10)$ & $  (t^{18}, t^{27} + t^{57} + t^{67})$  & $(18, 9_5, 3_6)$ & $1$& $\frac{1}{18}+\frac{1}{27}$  & $-2$ & \ref{thm:threepairs}, \S \ref{tonocurves} \\
    \hline 28 & $(6, 7), (2, 13), (2, 3)$ & $  (t^{24}, t^{28} + t^{54} + t^{57})$  & $(24, 4_{12}, 2_3)$ & $- \infty$ & $\frac{1}{24}+\frac{1}{28}$ 
 & $2$ & \ref{thm:threepairs}, \S \ref{amscurves} \\
    \hline 30 & $(2,3), (10,29) $  & $  (t^{20}, t^{30}+t^{59})$&$(20, 10_4, 9)$ & $- \infty$ & $\frac{1}{20}+\frac{1}{30}$  & $10$ & \ref{thm:twopairs}.3, \S \ref{amscurves} \\
    \hline 30 & $(3,11), (5,14) $  & $  (t^{15}, t^{55}+t^{69})$&$(15_3, 10, 5_4, 4)$& $- \infty$ & $\frac{1}{15}+\frac{1}{55}$  &  $5$ & \ref{thm:twopairs}.4, \S \ref{amscurves} \\
    \hline 30 & $(4,5), (6,29) $  & $  (t^{24}, t^{30}+t^{59})$&$(24, 6_8, 5)$ & $- \infty$& $\frac{1}{24}+\frac{1}{30}$  & $6$ & \ref{thm:twopairs}.3, \S \ref{amscurves} \\
    \hline 30 & $(5,6), (5,29) $  & $  (t^{25}, t^{30}+t^{59})$&$(25, 5_{10}, 4)$& $- \infty$& $\frac{1}{25}+\frac{1}{30}$  &  $5$& \ref{thm:twopairs}.3, \S \ref{amscurves} \\
    \hline 30 & $(5,19), (3,14) $  & $  (t^{15}, t^{57}+t^{71})$&$(15_3, 12, 3_8, 2)$ & $- \infty$ & $\frac{1}{15}+\frac{1}{57}$  & $3$ & \ref{thm:twopairs}.4, \S \ref{amscurves} \\
    \hline 30 & $(9,10), (3,29) $  & $  (t^{27}, t^{30}+t^{59})$&$(27, 3_{18}, 2)$& $- \infty$& $\frac{1}{27}+\frac{1}{30}$  & $3$ & \ref{thm:twopairs}.3, \S \ref{amscurves} \\
    \hline 30 & $(14,15), (2,29) $  & $  (t^{28}, t^{30}+t^{59})$&$(28, 2_{28})$  & $- \infty$ & $\frac{1}{28}+\frac{1}{30}$  & $2$ & \ref{thm:twopairs}.3, \S \ref{amscurves} \\
    \hline 30 & $(2, 3), (2, 5), (5, 9)$ & $  (t^{20}, t^{30} + t^{55} + t^{64})$  & $(20, 10_4, 5_3, 4)$ & $- \infty$ & $\frac{1}{20}+\frac{1}{30}$  & $5$ & \ref{thm:threepairs}, \S \ref{amscurves} \\
    \hline 30 & $(2, 3), (5, 14), (2, 9)$ & $  (t^{20}, t^{30} + t^{58} + t^{67})$  & $(20, 10_4, 8, 2_8)$ & $- \infty$ & $\frac{1}{20}+\frac{1}{30}$  & $2$ & \ref{thm:threepairs}, \S \ref{amscurves} \\
    \hline 30 & $(4, 5), (2, 9), (3, 5)$ & $  (t^{24}, t^{30} + t^{57} + t^{62})$  & $(24, 6_8, 3_3, 2)$ & $- \infty$& $\frac{1}{24}+\frac{1}{30}$  & $3$ & \ref{thm:threepairs}, \S \ref{amscurves} \\
    \hline 30 & $(4, 5), (3, 14), (2, 5)$ & $  (t^{24}, t^{30} + t^{58} + t^{63})$  &  $(24, 6_8, 4, 2_4)$ & $- \infty$ & $\frac{1}{24}+\frac{1}{30}$  & $2$ & \ref{thm:threepairs}, \S \ref{amscurves} \\
    \hline
\end{longtable}
}
\end{center}

\end{theorem}

\appendix 

\section{Computer code}\label{appendix}

Here is the program used to produce a list of candidate curves with cusp parameterized by three Newton pairs.

\begin{verbatim}
def tripleNewtonPairs(degree):
    ab_max = int(((degree - 1) * (degree - 2)) / 2)
    return generate_mn_list(ab_max, degree)

def generate_ab_list(ab_max):
    ab_list = []
    for a in range(2, ab_max+1):
        for b1 in range(a+1, ab_max+1):
            for b2 in range(b1+1, ab_max+1):
                for b3 in range(b2+1, ab_max+1):
                    ab_list.append((a, b1, b2, b3))
    return ab_list

def generate_mn_list(ab_max, degree):
    mn_list = []
    ab_list = generate_ab_list(ab_max)
    for nums in ab_list:
        #parameterization
        a = nums[0]
        b1 = nums[1]
        b2 = nums[2]
        b3 = nums[3]

        # uppercase: check delta invariant
        m1 = a # m_1 * m_2 * m_3
        m2 = math.gcd(m1, b1) # m_2 * m_3
        m3 = math.gcd(m2, b2) # m_3
        n1 = b1 # n_1 * m_2 * m_3
        n2 = b2 - b1 # n_2 * m_3
        n3 = b3 - b2 # n_3

        # lowercase: values in newton pairs (m, n)
        m_3 = m3
        m_2 = int(m2 / m3)
        m_1 = int(m1 / m2)
        n_3 = n3
        n_2 = int(n2 / m3)
        n_1 = int(n1 / m2)

        # semigroup generators
        w1 = m1
        w2 = n1
        w3 = (m_1 * w2) + n2
        w4 = (m_2 * w3) + n3

        # delta invariant
        d_inv = int(((degree - 1) * (degree - 2)))

        # check constraints 
            if all(i > 1 for i in (m1, m2, m3, m_1, m_2, m_3, n_1, n_2)):
                if (math.gcd(m_1, n_1) == 1 and math.gcd(m_2, n_2) == 1 
                and math.gcd(m_3, n_3) == 1):
                    pair1 = (m1, n1)
                    pair2 = (m2, n2)
                    pair3 = (m3, n3)
                    # check for duplicate pairs
                    if (not pair1 == pair2 and not 
                    pair1 == pair3 and not pair2 == pair3):
                        d_inv_check = int(((m1 - 1) * (n1 - 1)) + 
                        ((m2 - 1) * n2) + ((m3 - 1) * n3))
                        # check delta invariant
                        if(d_inv == d_inv_check):
                            valid = True
                            semigroup = [0]
                            # check semigroup property
                            for l in range(1, degree-2):
                                semigroup = getSemigroupTriple(w1, w2, w3, w4,
                                l, degree, semigroup)
                                compare = (l+1)*(l+2)/2
                                if(len(semigroup) != compare):
                                    valid = False
                            if(valid):
                                mn_list.append((pair1, pair2, pair3))
    return mn_list

def getSemigroupTriple(w1, w2, w3, w4, l, degree, semigroup):
    maxVal = l * degree
    for num in semigroup:
        num1 = num + w1
        num2 = num + w2
        num3 = num + w3
        num4 = num + w4
        if(num1 not in semigroup and num1 <= maxVal):
            semigroup.append(num1)
            semigroup.sort()
        if(num2 not in semigroup and num2 <= maxVal):
            semigroup.append(num2)
            semigroup.sort()
        if(num3 not in semigroup and num3 <= maxVal):
            semigroup.append(num3)
            semigroup.sort()
        if(num4 not in semigroup and num4 <= maxVal):
            semigroup.append(num4)
            semigroup.sort()
        if (num1 > maxVal and num2 > maxVal and num3 > maxVal and num4 > maxVal):
            return semigroup
\end{verbatim}


\bibliographystyle{alpha}
\bibliography{curves}

\newcommand{\etalchar}[1]{$^{#1}$}
\begin{thebibliography}{FdBLVMHN06}

\bibitem[ABB{\etalchar{+}}23]{BABWILD}
Kenneth Ascher, Dori Bejleri, Harold Blum, Kristin DeVleming, Giovanni
  Inchoistro, Yuchen Liu, and Xiaowei Wang.
\newblock {Moduli of boundary polarized Calabi-Yau pairs}.
\newblock 2023.
\newblock \href{https://arxiv.org/abs/2307.06522}{\textsf{arXiv:2307.06522}}.

\bibitem[ABIP23]{ABIP}
Kenneth Ascher, Dori Bejleri, Giovanni Inchiostro, and Zsolt Patakfalvi.
\newblock Wall crossing for moduli of stable log pairs.
\newblock {\em Ann. of Math. (2)}, 198(2):825--866, 2023.

\bibitem[ADL19]{ADL19}
Kenneth Ascher, Kristin DeVleming, and Yuchen Liu.
\newblock {Wall crossing for K-moduli spaces of plane curves}.
\newblock 2019.
\newblock \href{https://arxiv.org/abs/1909.04576}{\textsf{arXiv:1909.04576}}.

\bibitem[BK86]{PlaneAlgCurves}
Egbert Brieskorn and Horst Kn\"{o}rrer.
\newblock {\em Plane algebraic curves}.
\newblock Modern Birkh\"{a}user Classics. Birkh\"{a}user/Springer Basel AG,
  Basel, 1986.
\newblock Translated from the German original by John Stillwell, [2012] reprint
  of the 1986 edition.

\bibitem[BL14]{BorodzikLivingston}
Maciej Borodzik and Charles Livingston.
\newblock Heegaard {F}loer homology and rational cuspidal curves.
\newblock {\em Forum Math. Sigma}, 2:Paper No. e28, 23, 2014.

\bibitem[Bod16]{Bodnar2016}
J.~Bodn\'{a}r.
\newblock Classification of rational unicuspidal curves with two {N}ewton
  pairs.
\newblock {\em Acta Math. Hungar.}, 148(2):294--299, 2016.

\bibitem[BZ06]{BaierZhao}
Stephan Baier and Liangyi Zhao.
\newblock Bombieri-{V}inogradov type theorems for sparse sets of moduli.
\newblock {\em Acta Arith.}, 125(2):187--201, 2006.

\bibitem[DeV22]{KmoduliAGNESnotes}
Kristin DeVleming.
\newblock {K-moduli of Fano varieties and log Fano pairs}, 2022.
\newblock Available at \\
  \href{https://people.math.umass.edu/~devleming/research/introtokmoduli.pdf}{\textsf{https://people.math.umass.edu/$\sim$devleming/research/introtokmoduli.pdf}}.

\bibitem[DS22]{DS22}
Kristin DeVleming and David Stapleton.
\newblock {Smooth limits of plane curves of prime degree and {M}arkov numbers}.
\newblock 2022.
\newblock \href{https://arxiv.org/abs/2208.10595}{\textsf{arXiv:2208.10595}}.

\bibitem[FdBLMHN07]{FdBLMHN2}
Javier Fern\'{a}ndez~de Bobadilla, Ignacio Luengo, Alejandro
  Melle~Hern\'{a}ndez, and Andras N\'{e}methi.
\newblock Classification of rational unicuspidal projective curves whose
  singularities have one {P}uiseux pair.
\newblock In {\em Real and complex singularities}, Trends Math., pages 31--45.
  Birkh\"{a}user, Basel, 2007.

\bibitem[FdBLVMHN06]{FdBLVMHN1}
J.~Fern\'{a}ndez~de Bobadilla, I.~Luengo-Velasco, A.~Melle-Hern\'{a}ndez, and
  A.~N\'{e}methi.
\newblock On rational cuspidal projective plane curves.
\newblock {\em Proc. London Math. Soc. (3)}, 92(1):99--138, 2006.

\bibitem[Hwa00]{Hwang}
Hsien-Kuei Hwang.
\newblock Distribution of the number of factors in random ordered
  factorizations of integers.
\newblock {\em J. Number Theory}, 81(1):61--92, 2000.

\bibitem[Kal30]{Kalmar}
L\'azsl\'o Kalm\'ar.
\newblock Distribution of the number of factors in random ordered
  factorizations of integers.
\newblock {\em Acta Sci. Math. (Szeged)}, 5:95--107, 1930.

\bibitem[Kas87]{Kashiwara}
Hiroko Kashiwara.
\newblock Fonctions rationnelles de type {$(0,1)$} sur le plan projectif
  complexe.
\newblock {\em Osaka J. Math.}, 24(3):521--577, 1987.

\bibitem[KM98]{KollarMori}
J\'{a}nos Koll\'{a}r and Shigefumi Mori.
\newblock {\em Birational geometry of algebraic varieties}, volume 134 of {\em
  Cambridge Tracts in Mathematics}.
\newblock Cambridge University Press, Cambridge, 1998.
\newblock With the collaboration of C. H. Clemens and A. Corti, Translated from
  the 1998 Japanese original.

\bibitem[Kol23]{Kollar23}
J\'{a}nos Koll\'{a}r.
\newblock {\em Families of varieties of general type}, volume 231 of {\em
  Cambridge Tracts in Mathematics}.
\newblock Cambridge University Press, Cambridge, 2023.
\newblock With the collaboration of Klaus Altmann and S\'{a}ndor J. Kov\'{a}cs.

\bibitem[Liu14]{Liu2014}
Tiankai Liu.
\newblock {\em On planar rational cuspidal curves}.
\newblock ProQuest LLC, Ann Arbor, MI, 2014.
\newblock Thesis (Ph.D.)--Massachusetts Institute of Technology.

\bibitem[Moe08]{MoeThesis}
Torgunn~Karoline Moe.
\newblock Rational cuspidal curves.
\newblock 2008.
\newblock
  \href{https://www.duo.uio.no/handle/10852/10759}{\textsf{www.duo.uio.no/handle/10852/10759}}.

\bibitem[MS89]{MatsuokaSakai}
Takashi Matsuoka and Fumio Sakai.
\newblock The degree of rational cuspidal curves.
\newblock {\em Math. Ann.}, 285(2):233--247, 1989.

\bibitem[Oda13]{Odaka}
Yuji Odaka.
\newblock The {GIT} stability of polarized varieties via discrepancy.
\newblock {\em Ann. of Math. (2)}, 177(2):645--661, 2013.

\bibitem[Ore02]{Orevkov}
S.~Yu. Orevkov.
\newblock On rational cuspidal curves. {I}. {S}harp estimate for degree via
  multiplicities.
\newblock {\em Math. Ann.}, 324(4):657--673, 2002.

\bibitem[SI20]{oeis}
Neil J.~A. Sloane and The OEIS~Foundation Inc.
\newblock The on-line encyclopedia of integer sequences, {S}eq. {A}178444,
  2020.
\newblock \href{https://oeis.org/A178444}{\textsf{oeis.org/A178444}}.

\bibitem[Ton00]{Tono2000}
Keita Tono.
\newblock Defining equations of certain rational cuspidal curves. {I}.
\newblock {\em Manuscripta Math.}, 103(1):47--62, 2000.

\bibitem[Ton01]{Tono2001}
Keita Tono.
\newblock Rational unicuspidal plane curves with {$\overline\kappa=1$}.
\newblock Number 1233, pages 82--89. 2001.
\newblock Newton polyhedra and singularities (Japanese) (Kyoto, 2001).

\bibitem[Ton12]{Tono2012}
Keita Tono.
\newblock On {O}revkov's rational cuspidal plane curves.
\newblock {\em J. Math. Soc. Japan}, 64(2):365--385, 2012.

\bibitem[Tsu81]{Tsunoda}
Shuichiro Tsunoda.
\newblock The complements of projective plane curves.
\newblock In {\em Commutative algebra and algebraic geometry}, Proceedings of a
  {S}ymposium held at the {R}esearch {I}nstitute for {M}athematical {S}ciences,
  {K}yoto {U}niversity, {K}yoto, {S}eptember 10--12, 1981, pages 48--55. Kyoto
  University, Research Institute for Mathematical Sciences, Kyoto, 1981.

\bibitem[Xu23]{Xu23}
Chenyang Xu.
\newblock K-stability of {F}ano varieties, 2023.
\newblock Available at \\
  \href{https://web.math.princeton.edu/~chenyang/Kstabilitybook.pdf}{\textsf{https://web.math.princeton.edu/$\sim$chenyang/Kstabilitybook.pdf}}.

\end{thebibliography}

\end{document}